\documentclass[11pt]{article}

\usepackage{amsmath,amsfonts,amssymb,amsthm,enumerate}
\usepackage{latexsym,color,bbm,mathtools,mathrsfs}

\usepackage[utf8]{inputenc}
\usepackage[normalem]{ulem}

\usepackage{tikz, rotating}

\usepackage{geometry}
\usepackage{natbib}

\newtheorem{theorem}{Theorem}
\newtheorem{proposition}{Proposition}[section]
\newtheorem{lemma}[proposition]{Lemma}
\newtheorem{corollary}[proposition]{Corollary}

\theoremstyle{definition}
\newtheorem{definition}[proposition]{Definition}
\newtheorem{remark}[proposition]{Remark}
\newtheorem{example}[proposition]{Example}

\newcommand{\E}{\mathbbm{E}} \renewcommand{\P}{\mathbbm{P}}
\newcommand{\N}{\mathbbm{N}} 
\newcommand{\R}{\mathbbm{R}}

\newcommand{\eps}{\varepsilon} 

\newcommand{\abs}[1]{\lvert#1\rvert} 
\newcommand{\Abs}[1]{\left\lvert#1\right\rvert}
\newcommand{\norm}[1]{\lVert#1\rVert} 

\newcommand{\indset}[1]{\mathbbm{1}_{#1}} 

\DeclareMathAlphabet{\mathpzc}{OT1}{pzc}{m}{it}

\newcommand{\smallx}{\mathpzc{x}}
\newcommand\unnumberedfootnote[1]{ %
        \let\temp=\thefootnote %
        \renewcommand{\thefootnote}{}%
        \footnote{#1}%
        \let\thefootnote=\temp%
        \addtocounter{footnote}{-1}}

\numberwithin{equation}{section}

\title{A mixing tree-valued process arising under neutral evolution
  with recombination} \author{Andrej Depperschmidt, Etienne
  Pardoux, Peter Pfaffelhuber} \date{\today}

\begin{document}
\maketitle
\unnumberedfootnote{\emph{AMS 2010 subject classification.} {\tt
    92D15} (Primary) {\tt 60G10, 60K35} (Secondary).}

\unnumberedfootnote{\emph{Keywords and phrases.} Ancestral
  recombination graph, Kingman coalescent, tree-valued process,
  Gromov-Hausdorff metric}

\begin{abstract}
  The genealogy at a single locus of a constant size $N$ population in
  equilibrium is given by the well-known Kingman's coalescent. When
  considering multiple loci under recombination, the ancestral
  recombination graph encodes the genealogies at all loci in one graph.
  For a continuous genome $\mathbbm G$, we study the tree-valued
  process $(\mathscr T^N_u)_{u\in\mathbbm G}$ of genealogies along the
  genome in the limit $N\to\infty$. Encoding trees as metric measure
  spaces, we show convergence to a tree-valued process with c\`adl\`ag
  paths. In addition, we study mixing properties of the resulting
  process for loci which are far apart.
\end{abstract}

\section{Introduction}
\label{sec:introduction}
A large body of literature within the area of mathematical population
genetics is dealing with models for population of constant size. While
finite models such as the Wright-Fisher or the Moran model all have
their specificities, the limit of large populations -- given some
moments are bounded -- leads to a unified framework with diffusions
and genealogical trees as their main tools; see e.g.\
\cite{Ewens2004}. In finite population models of size $N$ --
frequently denoted Cannings models \citep{Cannings1974} --- the
offspring distribution of all individuals in each generation is
exchangeable and subject to the constraint of a constant population
size.

Neutral evolution accounts for the fact that all individuals have the
same chance to produce offspring in next generations. Recombination is
the evolutionary force by which genetic material from more than one
(i.e.\ two in all biologically relevant cases) parents is mixed in
their offspring. Genealogies under neutral evolution without
recombination are given through the famous Kingman coalescent
\citep{Kingman1982a}, a random binary tree where pairs of lines merge
exchangeably in a Markovian fashion. Genealogies under recombination
must deal with the fact that recombination events mix up genetical
material from the parents. As a consequence, lines not only merge due
to joint ancestry, but also split due to different ancestors for the
genetic material along the genome. The resulting genealogy is encoded
in the \emph{Ancestral Recombination Graph} (ARG), which appeared
already in \cite{Hudson1983}, but entered the mathematical literature
only in \cite{Griffiths1991, GriffithsMarjoram1997}. This graph gives
the genealogies of all genetic loci under stationarity at once; see
also Figure~\ref{fig:arg}.

The sequence of genealogies along the chromosome is most important for
biological applications, and fast simulation and inference of such
genealogies is a major research topic today
\citep{Rasmussen2014}. While the ARG gives the sequence of genealogies
from the present to the past, a construction of genealogies along the
chromosome is possible as well \citep{WiufHein1999,
   LeocardPardoux2010}. The advantage of the latter approach is that it
allows to approximate the full sequence by ignoring long-range
dependencies, a fruitful research topic started by \cite{McVean2005}.

The goal of the present paper is to study the sequence of genealogies
along the genome, denoted $\mathbbm G$, in the limit
$N\to\infty$. Precisely, we will use the notion of (ultra-) metric
measure spaces, introduced in the probabilistic community by
\cite{GrevenPfaffelhuberWinter2009}, in order to formalize
genealogical trees, read off the sequence $(\mathscr
T_u^N)_{u\in\mathbbm G}$ from the ARG and let $N\to\infty$. As main
results, we obtain convergence (Theorem~\ref{T1}) to an ergodic
tree-valued process which has c\`adl\`ag paths and study its mixing
properties (Theorem~\ref{T2}). We start by introducing our notation.

\begin{remark}[Notation]
  Let $(E,r)$ be a metric space. We denote by $\mathcal M_1(E)$ the
  space of all probability measures on $E$ equipped with the
  Borel-$\sigma$-algebra $\mathcal B(E)$. The space $\mathcal C_b(E)$
  consists of all continuous, bounded, real-valued functions defined
  on $E$. For a second metric space $(F,r_F)$ and
  $\mu \in \mathcal M_1(E)$ and a measurable map $\varphi: E\to F$,
  the measure $\varphi_\ast \mu$ is the push-forward of $\mu$ under
  $\varphi$. We denote vectors $(x_1, x_2,\dots) \in E^{\mathbbm N}$
  by $\underline x$ and integrals will be frequently denoted by
  $\langle \nu, f\rangle \coloneqq \int f d\nu$. Weak convergence of
  probability measures will be denoted by $\Rightarrow$. For
  $I\subseteq \R$, the space $\mathcal D_E(I)$ is the set of c\`adl\`ag
  functions $f:I\to E$.
\end{remark}

\section{Ancestry under recombination}
For a set of loci $\mathbbm G$, also called \emph{genome} in the
sequel, we aim to study the ancestry of individuals from a large
population. The joint genealogy for all loci is given by the ancestral
recombination graph (Section~\ref{S:21}), from which we can read off
genealogical trees at all loci $u\in\mathbbm G$
(Section~\ref{S:22}). In Section~\ref{S:23}, we formalize random
genealogies as metric measure spaces.

\subsection{The ancestral recombination graph}
\label{S:21}

In this section we give a formal definition of the \emph{ancestral
  recombination graph (ARG)} which is a (slight) generalization of the
definition from \cite{GriffithsMarjoram1997}; see also
the leftmost branching and coalescing graph in Figure~\ref{fig:arg}.

\begin{definition}[$N$-ancestral recombination graph\label{def:ARG}]
  \leavevmode
  \begin{enumerate}
  \item For $a<b$, $\mathbbm G \coloneqq [a,b]$, $\rho>0$ and a finite
    set $[N] \coloneqq \{1,\dots,N\}$, the $N$-\emph{ancestral
      recombination graph (ARG)}, denoted by
    $\mathcal A \coloneqq \mathcal A^{N} \coloneqq \mathcal
    A^{N}(\mathbbm G)$,
    starting with particles in the set $[N]$ is defined by the
    following Markovian dynamics:
    \begin{enumerate}[(i)]
    \item When there are $k \ge 2$ particles, two randomly chosen
      particles \emph{coalesce} (merge) at rate $\binom k 2$ and give
      rise to a single new particle.
    \item Each particle \emph{splits} in two at rate $\rho(b-a)$,
      resulting in a new \emph{left} and a new \emph{right} particle.
      Such a splitting event is marked by an independent, uniformly
      distributed random variable $U\in \mathbbm G$.
    \end{enumerate}
    Denote by $\mathcal A_t$ the set of particles at time $t\geq 0$
    and stop when there is only one particle left.
  \item The particle-counting process $\mathcal N = (N_t)_{t\geq 0}$
    with $N_t=\# \mathcal A_t$ for $\mathcal A$ is a birth-and death
    chain. Precisely, $\mathcal N$ has birth rate $b_k = \rho(b-a)k$
    and death rate $d_k = \binom k 2$, $k=1,2,\dots $ and is stopped at
    $T=\inf\{t: N_t=1\}$.
  \end{enumerate}
\end{definition}
\noindent
Since the birth rates are linear and the death rates are quadratic,
the expectation of the stopping time $T$ is finite; see Theorem~2.1
in \cite{PardouxSalamat2009} for an explicit expression.

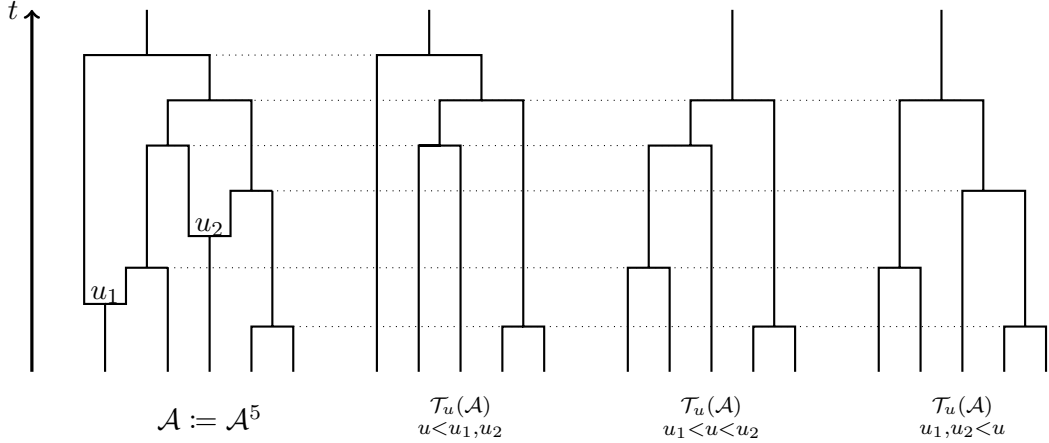
\begin{figure}
        \begin{tikzpicture}[xscale=.55,yscale=0.6]
          \draw[->,very thick] (-2.25,0) --(-2.25,8);
          \draw (-2.7,8) node {$t$};

          \draw[thick] (-0.5,0)--(-0.5,1.5)--(-1,1.5)--(-1,7)--(2,7) -- (2,6)
          (-0.5,1.5) -- (0,1.5) -- (0,2.3) -- (1,2.3)
          (1,0) -- (1,2.3)
          (2,0) -- (2,3) -- (1.5,3) -- (1.5,5) -- (0.5,5) -- (0.5,2.3)
          (1,5) -- (1,6) -- (3,6) -- (3,4)
          (2,3) -- (2.5,3) -- (2.5,4) -- (3.5,4)
          (3,0) -- (3,1) -- (4,1) -- (4,0)
          (3.5,1) -- (3.5,4) (0.5,7) -- (0.5,8);
          \draw (-0.5,1.7) node {$u_1$};
          \draw (2,3.2) node {$u_2$};
          \draw (2,-1) node {$\mathcal A \coloneqq \mathcal A^5$};

          \draw[thick] (6,0)--(6,7)--(8.5,7)--(8.5,6)--(7.5,6)--(7.5,5)--(7,5)--(7,0)
          (7.25,7)--(7.25,8)
          (8.5,6)--(9.5,6)--(9.5,1)--(10,1)--(10,0)
          (9.5,1)--(9,1)--(9,0)
          (7,5)--(8,5)--(8,0);
          \draw (8,-1) node {$\mathcal T_u(\mathcal A) \atop u < u_1,u_2$};

          \draw[thick] (12, 0) -- (12, 2.3) -- (13, 2.3) -- (13,0)
          (12.5, 2.3) -- (12.5, 5) -- (14, 5) -- (14,0)
          (13.5, 5) -- (13.5, 6) -- (15.5, 6) -- (15.5, 1)
          (15, 0) -- (15,1) -- (16,1) -- (16,0)
          (14.5, 6) -- (14.5, 8);
          \draw (14,-1) node {$\mathcal T_u(\mathcal A) \atop u_1 < u < u_2$};

          \draw[thick] (18, 0) -- (18, 2.3) -- (19, 2.3) -- (19,0)
          (18.5, 2.3) -- (18.5, 6) -- (20.5, 6) -- (20.5, 4)
          (20, 0) -- (20, 4) -- (21.5, 4) -- (21.5, 1)
          (21, 0) -- (21,1) -- (22,1) -- (22,0)
          (19.5, 6) -- (19.5, 8);
          \draw (20,-1) node {$\mathcal T_u(\mathcal A) \atop u_1, u_2 < u$};

          \draw[dotted] (4,1) -- (21,1)
          (1,2.3) -- (19, 2.3)
          (3, 4) -- (20,4)
          (1, 5) -- (14,5)
          (2, 6) -- (20,6)
          (2, 7) -- (8, 7);
        \end{tikzpicture}
        \caption{An example of an ARG $\mathcal A$ is given for $N =5$
          particles. From the resulting graph, where splitting events
          are marked by $u_1$ and $u_2$ (with $u_1<u_2$), trees can be
          read off by using right and left particle at these splitting
          events. This results in three (realizations of) $5$-Kingman
          coalescents.  \label{fig:arg}}
\end{figure}

\begin{remark}[Interpretation]
  Clearly, within the above definition, some biological interpretation
  can be given.
  \begin{enumerate}
  \item The set $\mathbbm G$ is the \emph{genome}, i.e.\ the set of
    all loci (for any individual within the population). An element
    $u\in\mathbbm G$ is called a \emph{locus}.
  \item The parameter $\rho$ is the \emph{recombination coefficient}
    per unit length.
  \item The set $[N]$ represents $N$ individuals sampled from a
    population and the particles within $\mathcal A_t$ form the
    ancestry of the individuals from $\mathcal{A}_0$ at time $t$ in
    the past.
  \item A coalescence event within $\mathcal A$ indicates joint
    ancestry.
  \item Instead of talking about \emph{left} and \emph{right}
    particles to follow within the ARG, the biological language would
    rather suggest to talk about \emph{upstream} and \emph{downstream}
    genomic sequences.
  \end{enumerate}
\end{remark}

\begin{remark}[ARG as a limiting graph, single crossovers]
  \leavevmode
  \begin{enumerate}
  \item The ARG arises as a limiting object within finite Moran models
    of population size $\widetilde N$ as $\widetilde N\to\infty$. In
    the model with recombination a finite population of $\widetilde N$
    individuals, each carrying a genetic material indexed by
    $\mathbbm G$, undergoes the following dynamics:
    \begin{enumerate}
    \item Every (unordered) pair of individuals \emph{resamples} at
      rate~1, that is, one individuals dies and is replaced by an
      offspring of the other individual. The offspring carries the
      same genetic material as the parent.
    \item At rate $\rho/\widetilde N$, every (unordered) pair $\{\ell,j\}$ of
      individuals \emph{resamples with recombination}, that is, a
      resampling event occurs, individual $j$ dies, say, a third
      individual $r$ and a random $U$, distributed uniformly on
      $\mathbbm G$, is chosen. Then, $j$ is exchanged by an
      individual, which carries genetic material $[a,U)$ from $\ell$
      and $[U,b]$ from $r$.
    \end{enumerate}
    When considering the history of a sample of $N \ll \widetilde N$
    individuals, two things can happen: First, if a resampling event
    of two individuals within the sample is hit, these individuals
    find a common ancestor, and their ancestral lines coalesces.
    Second, if a line hits a resampling event with recombination, the
    history of its genetic material is split at the corresponding $U$,
    and follows along two different lines. (Note that this happens at
    rate $(\widetilde N-1)\rho/\widetilde N \approx \rho$.) These two
    lines have a high chance to be outside the sample of $N$ lines if
    $N \ll\widetilde N$. As we see, as $\widetilde N\to\infty$, the
    ancestry is properly described by the ARG as in
    Definition~\ref{def:ARG}
  \item We assume here only \emph{single crossovers}, i.e.\ the mix of
    the genetic material of $\ell$ and $r$ is exactly as just
    described (rather than taking e.g.\ $[a,U_1] \cup (U_2, b]$ from
    $\ell$ and $(U_1, U_2]$ from $r$ for some random variables
    $a\leq U_1 < U_2\leq b$.)
  \end{enumerate}
\end{remark}

\subsection{Trees derived from an ARG}
\label{S:22}

In this section we describe sets of trees that can be read off from an
ARG and discuss some of their properties as well as different
constructions of the ARG. A construction of the ARG along the genome
(see Remark~\ref{rem:ARG_genome}) will be particularly useful in the
sequel and will be explained in more detail in
Section~\ref{sec:constr-along-genome}.

\begin{definition}[Genealogical trees read off from $\mathcal A$\label{def:Tu}]
  Let $\mathcal A = (\mathcal A_t)_{t\geq 0}$ be an ancestral
  recombination graph and $\mathcal B \subseteq [N]$ be a subset of
  the initial particles. For $u\in\mathbb G$ we read off the random
  tree
  $\mathcal T_u \coloneqq \mathcal T_u^{\mathcal B} \coloneqq \mathcal
  T_u^{\mathcal B}(\mathcal A)$
  (in the case $\mathcal B = [N]$, we also write
  $\mathcal T_u \coloneqq \mathcal T_u^{N} \coloneqq \mathcal
  T_u^{N}(\mathcal A)$)
  as follows:
  \begin{enumerate}[(i)]
  \item Start with particles in the set $\mathcal B$ and follow
  particles along $\mathcal A$.
  \item Upon coalescence events within $\mathcal A$, followed
    particles are merged as well. If a coalescence event within
    $\mathcal A$ only involves a single followed particle, continue to
    follow the coalesced particle.
  \item Upon a splitting event, consider its mark $U$. If $u\leq U$,
    follow the left particle, and if $u>U$, follow the right particle.
  \end{enumerate}
  We denote by
  $\mathcal T_{u,t} \coloneqq \mathcal T_{u,t}^{\mathcal B} \coloneqq
  \mathcal T_{u,t}^{\mathcal B}(\mathcal A)$
  the set of particles in $\mathcal T_u$ at time $t$. Furthermore we
  denote the root of $\mathcal T_u$ by $\bullet_u$.
\end{definition}

\begin{remark}[($N$-)coalescent\label{rem:Ncoal}]
  We will frequently use the notion of an $N$(-Kingman)-coalescent.
  This is a random tree arising by the following particle picture:
  Starting with $N$ particles, each pair of particles coalesces
  exchangeably at rate $1$. (Alternatively, we may say that the total
  coalescence rate when there are $k$ particles is $\binom k 2$ and
  upon a coalescence event, a random pair is chosen to coalesce.) The
  tree is stopped when reaching a single particle which we denote by
  $\bullet$ in the sequel. It is well-known (see also
  Example~\ref{ex:King}) that this random tree converges as
  $N\to\infty$ to the Kingman's coalescent.
\end{remark}

\begin{remark}[Properties of $\mathcal T_u^{\mathcal B}(\mathcal A)$]
  \leavevmode
  \begin{enumerate}
  \item Since $\mathcal A$ is stopped upon reaching a single particle,
    and no splits occur within $\mathcal T_u$, the latter is certainly
    a tree. In particular, its root may or may not be identical to the
    node of the stopping particle within $\mathcal A$.
  \item Note that for $M = \#\mathcal B$, each tree $\mathcal
    T_u^{\mathcal B}$ is an $M$-coalescent. Indeed, by exchangeability
    within Kingman's coalescent, any two particles within this tree
    coalesce at rate~1, independently of all others.
  \end{enumerate}
\end{remark}

\begin{remark}[Unused branches of ARG\label{rem:arg-branches}] If $R$
  is the number of recombination events in an ARG $\mathcal A$, then
  we can bound the number of different trees in $(\mathcal
  T^N_u(\mathcal A))_{u\in\mathbbm G}$. When following the left and
  right branches at each recombination point, we find $2^R$ different
  trees. However, since the $R$ recombination points have marks
  $U_1,\dots,U_R$, there are at most $R+1$ different trees arising from
  $u<U_{(1)}, U_{(i)} \leq u < U_{(i+1)}, i=1,\dots,R-1$ and $u\geq
  U_{(R)}$ (where $U_{(i)}$ is the $i$th order statistic of
  $U_1,\dots,U_R$). However, it is possible that a branch within
  $\mathcal A$ which is only followed when considering $u\in (v,b]$
  (for some $v\in\mathbbm G$) carries a recombination event with mark
  $U<v$. In this case, $\mathcal T^N_{v-} = \mathcal T^N_{v+}$ which
  reduces the number of different trees. For a lower bound of the
  number of different trees with $R$ recombination events in $\mathcal
  A$, we find a minimum of two different trees within $(\mathcal
  T^N_u(\mathcal A))_{u\in\mathbbm G}$ if $R>0$.

  This somewhat inefficient procedure of generating recombination
  events which do not take effect on the level of trees has the
  advantage of mathematical clarity and has been used by
  \cite{GriffithsMarjoram1997}. It is also possible to allow only
  recombination events which are used when reading off the trees
  $(\mathcal T^N_u(\mathcal A))_{u\in\mathbbm G}$;
  see~\cite{Hudson1983}. The latter procedure has the advantage of
  being more efficient in simulations.
\end{remark}

\begin{remark}[Construction of $(\mathcal T_u^N)_{u\in\mathbbm G}$
  along the genome]
  Instead of constructing the process $(\mathcal T_u^N)_{u\in\mathbbm
    G}$ from the present to the past along the ARG $\mathcal A$,
  \cite{WiufHein1999} have shown that there is also a construction
  along the genome. We will recall this approach together with
  approximations of $(\mathcal T_u^N)_{u\in\mathbbm G}$ related to
  this construction in Section~\ref{sec:constr-along-genome}.
\end{remark}

\begin{remark}[Outlook on Theorem~\ref{T1}\label{rem:T1}]
  Before we go on with introducing more objects needed to formulate
  our main results let us give an outlook on one of them.
  \begin{enumerate}
  \item Our goal is to study
    \begin{align}
      \label{eq:goal}
      \text{convergence of the process }(\mathcal
      T^{N}_u)_{u\in\mathbbm G} \text{ as $N \to\infty$.}
    \end{align}
    Since $\mathcal T^{N}_u$ is an $N$-coalescent for all
    $u\in\mathbbm G$, and as the convergence of the $N$-coalescent to
    Kingman's coalescent as $N\to\infty$ is well-known, convergence of
    finite-dimensional distributions in~\eqref{eq:goal} is not
    surprising. However, we will also show tightness of $\{(\mathcal
    T^{N}_u)_{u\in\mathbbm G}: N \in \mathbbm N\}$ in the space of
    c\`adl\`ag paths. This requires to define a proper topology on the
    space of trees, which we will do in the next section.
  \item In our formulation of Theorem~\ref{T1}, two different sets of
    trees derived from an ARG will arise:
    \begin{enumerate}
    \item For $\mathcal A^N$, we will consider
      $$\{\mathcal T_u^{N}(\mathcal A^N): u\in \mathbb G\},$$ which is the set of
      all trees with $N$ leaves from an $N$-ARG.
    \item For $\mathcal A^n$ and $n_1,\dots,n_j \in \mathbbm N$ with
      $n_1 + \cdots + n_j = n$, consider a partition $\{\mathcal B_1,
      \dots, \mathcal B_j\}$ of $[n]$ with $\# \mathcal B_1 =
      n_1,\dots, \#\mathcal B_j = n_j$ and $u_1,\dots,u_j \in \mathbbm
      G$. Then, we consider the trees
      $$ \{\mathcal T_{u_i}^{\mathcal B_i}(\mathcal A^n): i=1,\dots,j\}.$$
      These trees arise when considering an $n$-ARG and partition its
      initial particles into the sets $\mathcal B_1,\dots,\mathcal B_j$
      and following their ancestry. See
      Figure~\ref{fig:arg-neigb-trees} for an example of resulting
      trees and their  interaction with each other.
    \end{enumerate}
  \end{enumerate}
\end{remark}
\begin{figure}
  \centering
  \begin{tikzpicture}[xscale=0.9,yscale=0.9]
    \draw[very thick,black] (0,0) -- (0,1) -- (1,1) -- (1,0)
                            (0.5,1) -- (0.5,2.8)--(0,2.8)--(0,4) --
                            (0.5,4) -- (0.5,5) --(1.5,5) -- (1.5,5.6)
                            (2,0) -- (2,1.5) -- (2.45,1.5) --
                            (2.45,4.1) -- (3.5,4.1)--(3.5,5) -- (1.5,5);

    \draw[very thick,lightgray] (3,0) -- (3,1.5) -- (2.55,1.5) --
                                (2.55,4) -- (3.6,4) -- (3.6,5.1) -- (1.6,5.1)
                                (4,0)  -- (4,2.4) -- (2.69,2.4)
                                (2.31,2.4)--(0.6,2.4) --(0.6,2.8) --
                                (1,2.8) -- (1,4) -- (0.6,4) -- (0.6,4.9) --
                                (1.6,4.9) -- (1.6,5.6)
                                (5,0) -- (5,4) -- (3.6,4);
    \draw[very thick,lightgray] (2.3,2.4) to[out=45,in=135] (2.7,2.4);

    \draw (0,-.5) node {$1$};
    \draw (1,-.5) node {$2$};
    \draw (2,-.5) node {$3$};
    \draw (3,-.5) node {$4$};
    \draw (4,-.5) node {$5$};
    \draw (5,-.5) node {$6$};

  \end{tikzpicture}
  \caption{Two (interacting) trees with three leaves each read off
    from a joint ARG $\mathcal A^{6}$ starting with disjoint sets of
    leaves. The black tree (at locus $u=0$) is $\mathcal
    T_0^{\{1,2,3\}}(\mathcal A^6)$, while the gray tree (at locus
    $u=v$) is $\mathcal T_v^{\{4,5,6\}}(\mathcal A^6)$.}
  \label{fig:arg-neigb-trees}
\end{figure}
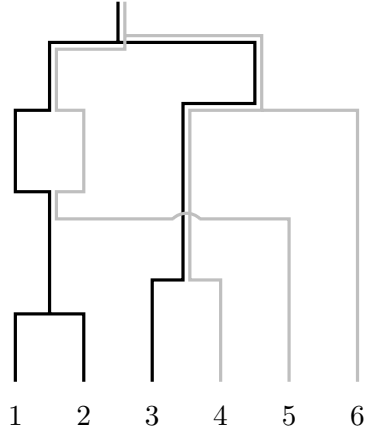

\subsection{Space of metric measure spaces}
\label{S:23}
Any (finite or infinite) genealogical tree can be encoded by an
(ultra-) metric space $(X,r)$ where $X$ is the set of leaves and $r$
is the genealogical distance. If we equip the set of leaves with a
sampling probability measure we obtain a \emph{metric measure space}
$(X,r,\mu)$. Spaces of metric measure spaces were introduced and their
topological properties were studied in \citet{Gromov2007}\footnote{The
  first edition of this book appeared in 1999. An even older french
  version from 1981 did not contain the chapter about metric measure
  spaces.} and \citet{GrevenPfaffelhuberWinter2009}. We now recall the
space of (isometry classes of) metric measure spaces $\mathbbm M$, the
Gromov-weak topology on $\mathbbm M$, and polynomials, which form a
convergence determining algebra of functions

\begin{definition}[mm-spaces]
  \leavevmode
  \begin{enumerate}
  \item A \emph{metric measure space (mm-space)} is a triple
    $(X,r,\mu)$ where $(X,r)$ is a complete and separable metric space
    and $\mu \in \mathcal M_1(X)$ with supp$(\mu)=X$. Two mm-spaces
    $(X_1,r_1,\mu_1)$ and $(X_2,r_2,\mu_2)$ are called
    \emph{measure-preserving isometric} if there exists an isometry
    $\varphi$ between $X_1$ and $X_2$ so that $\mu_2 = \varphi_\ast
    \mu_1$.
  \item Being measure preserving isometric is an equivalence relation
    and we denote the equivalence class of $(X,r,\mu)$ by
    $\overline{(X,r,\mu)}$ and write
    \begin{align}
      \label{eq:state-space}
      \mathbbm M \coloneqq \bigl\{\smallx= \overline{(X,r,\mu)} :
      (X,r,\mu) \text{ is a mm-space} \bigr\}
    \end{align}
    for the \emph{space of measure preserving isometry classes of
      mm-spaces}.
  \end{enumerate}
\end{definition}

In order to define a topology on $\mathbbm M$, we use the notion of
polynomials.

\begin{definition}[Distance matrix distribution, polynomials and the
  Gromov-weak topology]
  Let $(X,r,\mu)$ be a mm-space and $\smallx =
  \overline{(X,r,\mu)}$.
  \begin{enumerate}
  \item We define the function $R \coloneqq R^{(X,r)}: X^{\mathbbm N} \to
    \mathbbm R_+^{{\binom{\mathbbm N}{2}}}$ by $R(\underline x) =
    (r(x_i, x_j))_{1\leq i < j}$ and define the \emph{distance matrix
      distribution}
    $$ \nu^{\smallx}  = R_\ast \mu^{\otimes \mathbbm N} \in \mathcal{M}_1
    \Bigl(\R_+^{\binom \N 2}\Bigr).$$
  \item A function $\Phi: \mathbbm M \to \mathbbm R$ is called
    \emph{polynomial} if there is a bounded measurable function $\phi:
    \R_+^{\binom {\mathbbm N} 2} \to \R$, depending only on finitely
    many coordinates, such that for all $\smallx =
    \overline{(X,r,\mu)} \in \mathbbm M$ we have
    \begin{align}
      \label{eq:polynom}
      \Phi (\smallx) = \langle \mu^{\otimes \mathbbm N},\phi \circ
      R\rangle = \langle \nu^\smallx, \phi\rangle.
    \end{align}
    The smallest $n$ for which there is a function $\phi$ which
    depends only on coordinates $(r_{ij})_{1\leq i<j\leq n}$ so that
    \eqref{eq:polynom} holds is called the \emph{degree of the
      polynomial $\Phi$}. We then write $\Phi^{n,\phi}$ instead of
    $\Phi$ to stress the dependence on $n$ and $\phi$. The space of
    bounded continuous polynomials is denoted
    \begin{align}
      \Pi^0 \coloneqq \Big\{\Phi^{n,\phi}: n\in\mathbbm N, \phi \in \mathcal
      C_b(\mathbb R_+^{\binom{\mathbbm N}{2}})\Big\}.
    \end{align}
    (Sometimes we will abuse notation and write $\phi((r_{ij})_{1\leq
      i<j\leq n}) \coloneqq \phi((r_{ij})_{1\leq i<j})$.)
  \item The smallest topology on $\mathbbm M$ for which all functions
    $\Pi^0$ are continuous is called the \emph{Gromov-weak
      topology}. For this topology, $\smallx_n
    \xrightarrow{n\to\infty}\smallx$ if and only if $\nu^{\smallx_n}
    \xRightarrow{n\to\infty} \nu^\smallx$.
  \end{enumerate}
\end{definition}

\begin{remark}[Some properties of polynomials]
  \leavevmode
  \begin{enumerate}
  \item We stress that for $\smallx = \overline{(X,r\mu)}$, the
    measure $R_\ast \mu^{\otimes \mathbbm N}$ does not depend on the
    representative and hence $\nu^\smallx$ is well-defined.
  \item Given two polynomials $\Phi^{n,\phi}$ and $\Psi^{m,\psi}$ one
    can show that the product $\Phi^{n,\phi} \cdot \Psi^{m,\psi}$ is a
    polynomial of degree $n+m$; see Remark~2.8(i) in
    \cite{GrevenPfaffelhuberWinter2013}. The space $\Pi^0$ is an
    algebra which separates points; see Section~$3\frac12.5.$ in
    \cite{Gromov2007} and Proposition~2.6 in
    \cite{GrevenPfaffelhuberWinter2009}.
  \item The space $\mathbbm M$ equipped with the Gromov-weak topology
    is Polish. Later in Section \ref{S:31} we will give the
    Gromov-Prohorov metric on $\mathbbm M$, which is complete and
    metrizes the Gromov-weak topology (see Theorem~5 and
    Proposition~5.6 in \cite{GrevenPfaffelhuberWinter2009}). We will
    also give the Gromov-Hausdorff metric and other metrics on
    $\mathbbm M$, which we will need in the formulation and proof of
    Theorem~\ref{T1}.
  \end{enumerate}
\end{remark}

\begin{example}[Kingman coalescent tree as a metric measure
  space\label{ex:King}]
  Consider \label{ex:kingman-tree} the $N$-coalescent from
  Remark~\ref{rem:Ncoal}. Let $K^N=[N]=\{1,\dots,N\}$ be the set of
  leaves and let the metric $r^N$ be the usual tree distance, i.e.\
  $r^N(i,j)$ is twice the time to the MRCA of $i$ and $j$,
  $1\leq i,j\leq N$. Finally, let $\mu^N$ be the uniform measure on
  $K^N$. Then $\mathpzc K^N \coloneqq \overline{(K^N,r^N,\mu^N)}$ is
  an equivalence class of a metric measure space. Furthermore, by
  Theorem~4 in \cite{GrevenPfaffelhuberWinter2009} there exist an
  $\mathbbm M$-valued random variable $\mathpzc K^\infty$ such that
  \begin{align}
    \label{eq:kingman-tree}
    \mathpzc K^N \xRightarrow{N\to\infty} \mathpzc{K}^\infty.
  \end{align}
  The limiting object $\mathpzc{K}^\infty$ is called the \emph{Kingman
    measure tree}.
\end{example}

\subsection{Metrics on metric (measure) spaces}
\label{S:31}
We now recall the definitions of several distances on $\mathbbm M$
that we will use in the sequel.  While the Hausdorff distance is a
metric on closed subsets of a metric space, the Prohorov and total
variation distances are metrics on probability measures on a metric
space. All three distances can be turned into distances on $\mathbbm
M$.

\begin{definition}[Hausdorff, Prohorov and total variation
  distances\label{def:HausPro}]
  Let $(Z,d)$ be a metric space and $\mu_1, \mu_2 \in\mathcal M_1(Z)$.
  \begin{enumerate}
  \item The \emph{Hausdorff distance} between two subsets $A, B
    \subseteq Z$ is defined by
    \begin{align}
      \label{eq:d-H}
      d_{\mathrm H}(A,B) \coloneqq \inf\{\varepsilon>0: A\subseteq
      B^\eps, B\subseteq A^\eps\}.
    \end{align}
    where for $A\subseteq Z$\,
    $$ A^\eps\coloneqq \{ x \in Z: d(x,A) < \eps \} = \{ z \in Z:
    \exists y \in A, d(y,z) < \eps\}.$$
  \item The \emph{Prohorov distance} of $\mu_1, \mu_2$ is defined by
    \begin{align}
      \label{eq:pr}
      d_{\mathrm{P}}(\mu_1,\mu_2) \coloneqq \inf\{\eps >0 : \mu_1(F)
      \le \mu_2(F^\eps) + \eps, \forall F \subseteq Z, \text{
        closed}\, \}.
    \end{align}
  \item The \emph{total variation distance} of $\mu_1, \mu_2$ is
    defined by
    \begin{align}
      \label{eq:tv}
      d_{\mathrm{TV}}(\mu_1,\mu_2) \coloneqq \sup_{A \in \mathcal B(Z)}
      \abs{\mu_1(A) - \mu_2(A)}.
    \end{align}
  \end{enumerate}
\end{definition}

\begin{remark}[Total variation distance\label{rem:dTV}]
  \leavevmode
  \begin{enumerate}
  \item If $Z$ is finite, the total variation distance of probability
    measures $\mu_1$ and $\mu_2$ on $Z$ is given by
    \begin{align}
      \label{eq:tv-finite}
      d_{\mathrm{TV}}(\mu_1,\mu_2) = \frac12 \sum_{z \in Z}
      \, \abs{\mu_1(\{z\}) -\mu_2(\{z\})}.
    \end{align}
  \item Recall that the Prohorov distance of two probability measures
    is bounded by their total variation distance.
  \end{enumerate}
\end{remark}

\noindent
For all three notions just defined, we now recall the corresponding
``Gromov-versions'' which are (semi-) metrics on $\mathbbm M$. The
idea is always to find an optimal isometric embedding into a third
metric space and compute there the usual distance of the images of the
spaces and measures.

\begin{definition}[Gromov distances\label{def:Grom}]
  Let $\smallx_1 = \overline{(X_1, r_1, \mu_1)}$ and $\smallx_2 =
  \overline{(X_2, r_2, \mu_2)}$ be mm-spaces. Moreover, let $\varphi_1:
  X_1 \to Z$ and $\varphi_2: X_2\to Z$ be isometric embeddings into a
  common complete and separable metric space $(Z,d)$.
  \begin{enumerate}
  \item The \emph{Gromov-Hausdorff distance} of $\smallx_1$ and
    $\smallx_2$ is defined by
    \begin{align*}
      d_{\mathrm{GH}}(\smallx_1, \smallx_2) \coloneqq \inf_{\varphi_1,\varphi_2,Z}
      d_{\mathrm H}(\varphi_1(X_1),\varphi_2(X_2)).
    \end{align*}
  \item The \emph{Gromov-Prohorov metric} of $\smallx_1$ and
    $\smallx_2$ is defined by
    \begin{align*}
      d_{\mathrm{GP}}(\smallx_1, \smallx_2) \coloneqq \inf_{\varphi_1,\varphi_2,Z}
      d_{\mathrm{P}}((\varphi_1)_\ast \mu_1, (\varphi_2)_\ast \mu_2).
    \end{align*}
  \item The \emph{Gromov total variation distance} of $\smallx_1$ and
    $\smallx_2$ is defined by
    \begin{align*}
      d_{\mathrm{GTV}}(\smallx_1,\smallx_2) \coloneqq \inf_{\varphi_1,\varphi_2,Z}
      d_{\mathrm{TV}}((\varphi_1)_\ast \mu_1, (\varphi_2)_\ast \mu_2).
    \end{align*}
  \end{enumerate}
\end{definition}

\begin{remark}[Properties\label{rem:TC_properties}]
  Let us recall or rather state some known and some obvious properties
  of the distances just introduced.
  \begin{enumerate}
  \item The distances
    $d_{\mathrm{GH}}, d_{\mathrm{GP}}, d_{\mathrm{GTV}}$ are
    well-defined. As can be seen by considering isometries between
    elements of one isometry class the distances do not depend on the
    representative.
  \item Since the Gromov-Hausdorff distance only uses the metric
    spaces $(X_1, r_1)$ and $(X_2, r_2)$ but not the measures $\mu_1,
    \mu_2$, it is only a pseudo-metric on $\mathbbm M$.
  \item According to Lemma~5.4 and Proposition~5.6 in
    \citet{GrevenPfaffelhuberWinter2009} the Gromov-Prohorov metric
    $d_{\mathrm{GP}}$ is indeed a metric on $\mathbbm M$ and the
    metric space $(\mathbbm M,d_{\mathrm{GP}})$ is complete and
    separable. Moreover, it metrizes the Gromov-weak topology by
    Theorem~5 of \citet{GrevenPfaffelhuberWinter2009}.
  \item Since the Prohorov distance is bounded by the total variation
    distance, we also find that
    \begin{align}
      \label{eq:grpr-grtv}
      d_{\mathrm{GP}}(\smallx_1,\smallx_2) \le
      d_{\mathrm{GTV}}(\smallx_1,\smallx_2).
    \end{align}
    This will be useful later since the total variation distance is
    usually easier to compute or estimate than the Prohorov distance.
  \end{enumerate}
\end{remark}

\section{Main results}
We now formalize the trees $\{\mathcal T_u^N(\mathcal A^N):
u\in\mathbbm G\}$ as mm-spaces. Our results, then, are dealing with
these $\mathbbm M$-valued random processes. In particular,
Theorem~\ref{T1} studies convergence as $N\to\infty$. Since $\mathcal
T_u^N(\mathcal A^N)$ is an $N$-coalescent for all $u\in\mathbbm G$,
the resulting process is stationary. It is even mixing, as
Theorem~\ref{T2} shows.

\begin{definition}[$\mathcal T_u$ as an mm-space\label{def:Tumm}]
  Let $\mathcal A = \mathcal A^N(\mathbb G)$ for $N\in\mathbbm N$ and
  $\mathbbm G = [a,b]$ with $a<b$ be an $N$-ARG. For $u\in\mathbbm G$
  and $\mathcal B \subseteq [N]$, let $\mathcal T_u^{\mathcal
    B}(\mathcal A)$ be as in Definition~\ref{def:Tu}. As in
  Example~\ref{ex:King}, let $r$ be the usual tree-distance and $\mu$
  the uniform measure on $\mathcal B$. Then, $(\mathcal B, r, \mu)$ is
  a metric measure space. Its isometry class will be denoted $\mathpzc
  T_u \coloneqq \mathpzc T_u^{\mathcal B} \coloneqq \mathpzc T_u^{\mathcal
    B}(\mathcal A)$ in the sequel. If $\mathcal B = [N]$ we write
  $\mathpzc T_u \coloneqq \mathpzc T_u^{N} \coloneqq \mathpzc T_u^{N}(\mathcal A)$.
\end{definition}

\noindent
In Theorem~\ref{T1}, we need the notion of the variation of a function
which we briefly recall.

\begin{remark}[Variation]
  Let $(E,r)$ be a metric space and $f:I \to E$ for $I \subset
  \R$. The \emph{variation of $f$ with respect to $r$} on
  subintervals $[a,b] \subset I$ is defined by
  \begin{align}
    \label{eq:TV-f}
    \mathrm{V}_a^b (f) \coloneqq \sup \Bigl\{ \sum_{i=1}^k
    r(f(t_i),f(t_{i-1})) : \; k \in \N, \; a = t_0 < t_1 < \dots <
    t_k = b  \Bigr\}.
  \end{align}
\end{remark}

\begin{theorem}[Convergence of $N$-ARGs\label{T1}]
  Let $\mathpzc T^N\coloneqq(\mathpzc T_u^{N}(\mathcal A))_{u\in\mathbb G}$
  be as in Definition~\ref{def:Tumm}. Then, $\mathpzc T^N
  \xRightarrow{N\to\infty} \mathpzc T$ on $\mathcal D_{\mathbbm
    M}(\mathbbm G)$ for some process $\mathpzc T$. The (law of the)
  process $\mathpzc T = (\mathpzc T_u)_{u\in\mathbbm G}$ is
  uniquely given as follows:\\
  For each $j\in\mathbbm N$, $u_1,\dots,u_j\in\mathbbm G$, $n_1,\dots,n_j
  \in \mathbbm N$, let $\mathcal T_{u_i}^{\mathcal B_i}$ be as in
  Remark~\ref{rem:T1} and $\underline{\underline R}_i$ be the
  distances of leaves $\mathcal B_i$ within $\mathcal
  T_{u_i}^{\mathcal B_i}$. Then, for $\Phi_i = \Phi^{n_i, \phi_i}
  \in\Pi^0, i=1,\dots,j$,
  \begin{align}\label{eq:T11}
    \mathbb E[\Phi_1(\mathpzc T_{u_1}) \cdots \Phi_j(\mathpzc T_{u_j})
    ] = \mathbb E[\phi_1(\underline{\underline R}_1)
    \cdots\phi_j(\underline{\underline R}_j)].
  \end{align}
  The paths of $\mathpzc T$ are almost surely of finite variation with
  respect to Gromov-Prohorov, Gromov total variation and
  Gromov-Hausdorff metrics.
\end{theorem}

\begin{remark}[Path-properties of $\mathpzc T$\label{rem:pathprop}]
  We can ask about path-properties of the limiting process $\mathpzc
  T$. Let us give an example: Let $N_{u}^\varepsilon$ be the number of
  $\varepsilon$-balls that are needed to cover $\mathpzc T_u$. In
  other words $N_{u,}^\varepsilon$ is the number of families in
  $\mathpzc T_{u}$ some distance $\varepsilon$ from the leaves. For
  fixed $u\in \mathbbm G$, we know that $\mathpzc T_u$ is a Kingman
  coalescent and hence we have $\varepsilon N_{u}^\varepsilon
  \xrightarrow{\varepsilon\to 0} 2$ almost surely. (See e.g.\ (35) in
  \cite{Aldous1999}.) Is it also true that
  \begin{align*}
    \mathbb P (\varepsilon N_{u}^\varepsilon
    \xrightarrow{\varepsilon\to 0} 2 \; \text{for all $u \in \mathbbm
      G$} ) =1?
  \end{align*}
  Similar questions arise for well-known almost sure properties of
  Kingman's coalescent regarding the family-sizes near the leaves; see
  \cite{Aldous1999}, Chapter~4.2.
\end{remark}

\noindent
For our next result, we need to extend the ARG to $\mathbbm G =
(-\infty, \infty)$. This can be done using some projective property
(see also Lemma~\ref{l:proj2}): Let $\mathbbm G_n \uparrow (-\infty,
  \infty)$. Clearly, \eqref{eq:T11} gives the finite-dimensional
  distributions of $(\mathpzc T_u)_{u\in\mathbbm G_n}$ for every
  $n$. In particular, for $m<n$, we see that the projection of
  $(\mathpzc T_u)_{u\in\mathbbm G_n}$ to $\{u\in\mathbbm G_m\}$ is the
  same as $(\mathpzc T_u)_{u\in\mathbbm G_m}$ and therefore
  \eqref{eq:T11} defines a projective family of probability measures
  which can by Kolmogorov's extension Theorem be extended to the law
  of $(\mathpzc T_u)_{u\in\mathbbm R}$.

\begin{theorem}[Mixing properties\label{T2}]
  For $n\in\mathbbm N$ let $\Psi=\Psi^{n,\psi}$ and $
  \Phi=\Phi^{n,\phi}$ be polynomials of (at most) degree $n$ and
  $(\mathpzc T_u)_{u\in\mathbbm G}$ be the extension of the limit
  process $\mathpzc T$ from Theorem~\ref{T1} to $\mathbbm G = \mathbbm
  R$. Then, there exists a positive finite constant $C=C(n)$ such that
  \begin{align}
    \label{eq:cov-bound}
    \Abs{\E[\Psi(\mathpzc T_0) \Phi(\mathpzc T_u)] - \E[\Psi(\mathpzc
      T_0)] \E[\Phi(\mathpzc T_u)]} \le
    \frac{C}{\rho^2 u^2} \norm{\psi}_\infty \norm{\phi}_\infty , \quad
    u>0.
  \end{align}
\end{theorem}

\begin{remark}[Dependency on $n$ \label{rem:n4}]
  In our proof we will show that \eqref{eq:cov-bound} holds with
  $C(n)=\frac{2n^4}{9 + 7\rho v + \rho^2 v^2}$.  Therefore, the bound
  is only useful in the limit $u\to\infty$ for fixed $n$. If $n \to
  \infty$ and $u$ is fixed, the trivial bound $ \Abs{\E[\Psi(\mathpzc
    T_0) \Phi(\mathpzc T_u)] - \E[\Psi(\mathpzc T_0)] \E[\Phi(\mathpzc
    T_u)]} \le 2\norm{\psi}_\infty \norm{\phi}_\infty$ holds as
  well. It would be desirable to obtain a bound similar to
  \eqref{eq:cov-bound} uniformly in $n$ and $u$, but this seems to be
  out of reach with the techniques we develop here.
\end{remark}

\section{Preliminaries}
\label{S:3}

\subsection{Construction of $(\mathpzc T_u^{N})_{u\in\mathbbm G}$
  along the genome}
\label{sec:constr-along-genome}
In \cite{WiufHein1999}, a construction of an $N$-ARG was given, which
results in the same trees $(\mathcal T_u^N)_{u\in\mathbbm G}$ (or
$(\mathpzc T_u^N)_{u\in\mathbbm G}$) in distribution as described in
Definition~\ref{def:Tu} (and Definition~\ref{def:Tumm}), but
constructs this tree-valued process \emph{along the genome}, i.e.\ by
starting at $u \in \mathbbm G$ and then letting the trees evolve when
moving along $\mathbbm G$. This construction which we recall below
will be helpful in all further proofs. See also
\cite{LeocardPardoux2010}.

\begin{definition}[$N$-ARG' along the genome\label{def:ARG'}]
  \leavevmode
  \begin{enumerate}
  \item For $a<b$, $\mathbbm G = [a,b]$ and $N\in\mathbbm N$,
    construct an evolving pair $(\mathcal G, \mathcal T) = (\mathcal
    G_n, \mathcal T_n)_{n=0,1,2,\dots}$, where $\mathcal G_n$ is a graph
    and $\mathcal T_n$ is a tree as follows:
  \item Start with an $N$-coalescent $\mathcal G_0 = \mathcal T_0$
    with set of leaves $[N]$ (which is continued indefinitely, i.e.\
    not stopped upon hitting a single line) and in each step do the
    following (with $U_0=a$):
    \begin{enumerate}
    \item Measure the length of the (vertical branches of the) graph
      $L_n = L(\mathcal G_n)$, choose a uniformly distributed point
      $X_n \in \mathcal G_n$ and an independent exponential random
      variable $\xi_n$.
    \item From $X_n$, change the graph as follows: Create a split
      event at $X_n$, and mark it by $U_n = U_{n-1} + \xi_n/(L_n
      \rho)$. The line in $\mathcal G_n$ which starts at $X_n$ is
      called the \emph{left branch} and a new \emph{right branch} is
      created. This new branch coalesces with all other branches in
      $\mathcal G_n$ at rate~1. The resulting graph is $\mathcal
      G_{n+1}$ and $\mathcal T_{n+1}$ is given by following $\mathcal
      T_n$ until the right branch at $X_n$ is hit and waiting until
      this branch coalesces back with $\mathcal T_n$. (There is a
      chance that $X_n \notin \mathcal T_n$; in this case we have
      $\mathcal T_{n+1} = \mathcal T_n$.)
    \end{enumerate}
    Stop when $U_n >b$ and set $\mathcal G = \mathcal G_n$.
  \item Let us now consider the final graph $\mathcal G$. This is a
    coalescing-splitting random graph (similar to $\mathcal A$ in
    Definition~\ref{def:ARG}) and therefore can also be considered as
    an evolving set of particles which coalesce and split, where
    splitting events are marked by some element of $\mathbbm G$ and
    are continued by a \emph{left} and \emph{right} branch. Hence, we
    can define for a subset $\mathcal B\subseteq [N]$ and
    $u\in\mathbbm G$ the random tree $\mathcal S_u \coloneqq \mathcal
    S_u^{\mathcal B} \coloneqq \mathcal S_u^{\mathcal B}(\mathcal G)$
    as in Definition~\ref{def:Tu}. Again, we set $\mathcal S_u^N
    \coloneqq \mathcal S_u^{[N]}$.
  \end{enumerate}
\end{definition}

\begin{remark}[Properties of $(\mathcal S_u^{\mathcal
    A_0})_{u\in\mathbbm G}$\label{rem:PropS}]
  \leavevmode
  \begin{enumerate}
  \item From \cite{WiufHein1999}, the graphs $\mathcal A$ and
    $\mathcal G$ have the same distribution, hence the same is true
    for the processes $(\mathcal T_u^{N})_{u\in\mathbbm G}$ and
    $(\mathcal S_u^{N})_{u\in\mathbbm G}$.
  \item Let $\mathcal L\subseteq \mathcal S_v^{N}$ be a line of length
    $\ell$ starting in some leaf $x\in [N]$ and reaching in the
    direction of the root. Then,
    $\mathcal L\subseteq \mathcal S_w^{N}$ for some $w\in\mathbbm G$
    if and only if no recombination marks $\mathcal L$ before reaching
    $w$. By construction, this happens with probability
    $e^{-\rho|w-v|\ell}$. Moreover, let $u<v<w$. Then, (given
    $\mathcal S_v^{N}$), $\mathcal L \subseteq \mathcal S_u^{N}$ with
    probability $e^{-\rho|v-u|\ell}$, independent of the event
    $\mathcal L \subseteq \mathcal S_w^{N}$.
  \item By construction, $\mathcal G$ from above is an ARG and
    therefore, its length has finite expectation. Hence, the
    construction above terminates almost surely.
  \end{enumerate}
\end{remark}

\begin{remark}[Approximating the ARG \label{rem:ARG_genome} using a
  Markov process along the genome]
  One striking feature of the construction of \cite{WiufHein1999} is
  that it allows to approximate $(\mathcal T^N_u)_{u\in\mathbbm G}$ by
  a Markov process by changing the dynamics of the construction
  $(\mathcal G_n, \mathcal T_n)$ in order to obtain a Markovian
  dynamics $\mathcal T = (\mathcal T_n)$:
  \begin{enumerate}
  \item The following approximation was used by \cite{McVean2005}:
    Instead of (a) in Definition~\ref{def:ARG'}, measure the length of
    the (vertical branches of the) tree $L_n = L(\mathcal T_n)$,
    choose a uniformly distributed point $X_n \in \mathcal T_n$ and an
    independent exponential random variable $\xi_n$. Then, instead of
    (b), change the graph as follows: Create a split event at $X_n$,
    and mark it by $U_n = U_{n-1} + \xi_n/(L_n \rho)$. Delete the
    branch which connects $X_n$ to its ancestral node from $\mathcal
    T_n$, all other lines are available for coalescence. Then, start a
    new branch in $X_n$ which coalesces with all other available
    branches rate~1. The resulting tree is $\mathcal T_{n+1}$.  In
    particular, by only allowing coalescences with $\mathcal T^N_u$,
    this approximation becomes a Markov process, also called the
    Sequentially Markov Coalescent (SMC). (In SMC', \cite{Chen2009}
    use almost the same construction but without deleting the branch
    connecting $X_n$ to its ancestral node for the set of lines
    available for coalescence, leading to a better approximation.)
  \item Another simulation software based on the Markovian Coalescent
    Simulator (MaCS) from ~\cite{Chen2009}, has all lines from the
    last $k$ genealogical trees as \emph{available} for coalescence
    after a recombination event.
  \end{enumerate}
  While approximations such as SMC and MaCS make genome-wide computer
  simulations under recombination feasible, their construction differs
  from the ARG at least for loci which are far apart. In particular,
  Theorem~\ref{T2} would not be true for these approximations.
\end{remark}

\subsection{Conditional distances of trees and first upper bounds}

To obtain useful bounds on expected tree distances introduced in the
previous section we will condition on one of the trees and use the
construction of the tree-valued process along the genome from
Section~\ref{sec:constr-along-genome}. Let us start with an
illustrative example.

\begin{figure}
  \centering
  \begin{tikzpicture}[xscale=0.7,yscale=0.7]
    \draw[very thick] (-1,0)--(-1,7)--(1.5,7) -- (1.5,5) (0,0) -- (0,2.3)
    -- (1,2.3) (1,0) -- (1,2.3) (2,0) -- (2,1) -- (2.5,1) -- (2.5,5)
    -- (0.5,5) -- (0.5,2.3) (3,0) -- (3,1)--(2.5,1) (0.5,7) --
    (0.5,7.5);

    \filldraw (2.5,4) circle (2pt);
    \draw[thick,dashed] (2.5,4) -- (3.5,4) -- (3.5,6) -- (-1,6);

    \draw[thick,dotted] (-2.5,7) -- (-1,7) (-2.5,5) -- (1.5,5)
    (-2.5,2.3) -- (1,2.3) (-2.5,1) -- (2,1) (-2.5,0) -- (3,0);

    \draw[stealth-,thick] (-2.5,7) -- (-2.5,6.3);
    \draw[-stealth,thick] (-2.5,5.75) -- (-2.5,5);
    \draw (-2.5,6) node  {$S_2$};

    \draw[stealth-,thick] (-2.5,5) -- (-2.5,3.9);
    \draw[-stealth,thick] (-2.5,3.4) -- (-2.5,2.3);
    \draw (-2.5,3.65)  node {$S_3$};

    \draw[stealth-,thick] (-2.5,2.3) -- (-2.5,1.9);
    \draw[-stealth,thick] (-2.5,1.4) -- (-2.5,1);
    \draw (-2.5,1.65) node {$S_4$};

    \draw[stealth-,thick] (-2.5,1) -- (-2.5,0.7);
    \draw[-stealth,thick] (-2.5,0.3) -- (-2.5,0);
    \draw (-2.5,0.5) node {$S_5$};

    \draw (-1,-0.25) node {$x_1$};
    \draw (0,-0.25) node {$x_2$};
    \draw (1,-0.25) node {$x_3$};
    \draw (2,-0.25) node {$x_4$};
    \draw (3,-0.25) node {$x_5$};

    \draw[very thick, lightgray] (7,0) -- (7,1) -- (7.5,1) -- (7.5,6) -- (8.15,6)
                        (8,0) -- (8,1) -- (7.5,1)
                        (8.9,0)--(8.9,5.9) -- (8.15,5.9) -- (8.15,7.1)
                        --(11.6,7.1) -- (11.6,4.9) -- (10.6,4.9) -- (10.6,2.4)
                        (9.9,0) -- (9.9,2.4) -- (11.1,2.4) -- (11.1,0)
                        (10.4,7.1) -- (10.4,7.5);

    \draw[very thick, black] (9,0)--(9,6) -- (8.25,6) -- (8.25,7) --(11.5,7) -- (11.5,5)
                 (10,0) -- (10,2.3) -- (11,2.3)
                 (11,0) -- (11,2.3)
                 (12,0) -- (12,1) -- (12.5,1) -- (12.5,5) -- (10.5,5) -- (10.5,2.3)
                 (13,0) -- (13,1)--(12.5,1)
                 (10.5,7) -- (10.5,7.5);

    \draw (7,-0.3) node {$x_4'$};
    \draw (8,-0.3) node {$x_5'$};

    \draw (9,-0.3) node {$x_1$};
    \draw (10,-0.3) node {$x_2$};
    \draw (11,-0.3) node {$x_3$};
    \draw (12,-0.3) node {$x_4$};
    \draw (13,-0.3) node {$x_5$};
  \end{tikzpicture}
  \caption{A conditional embedding of the trees $\mathcal S^5_{u}$ and
    $\mathcal S^5_{v}$ into a common tree.}
  \label{fig:distance-est}
\end{figure}
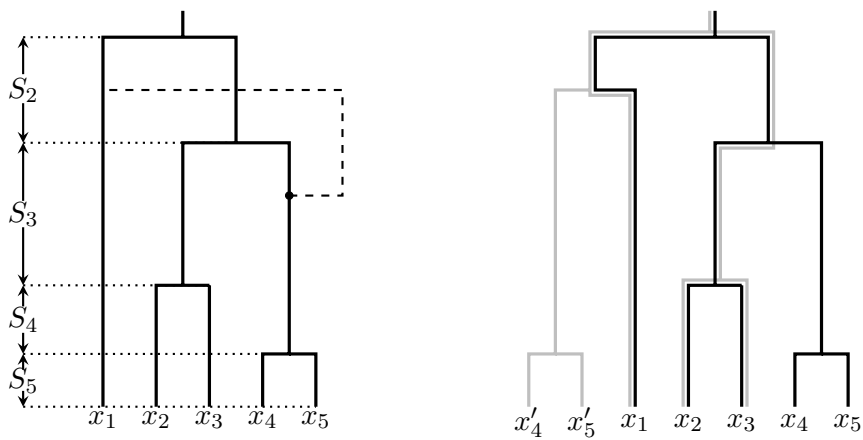

\begin{example}[A conditional embedding of two trees and upper
  bounds; Figure~\ref{fig:distance-est}]
  Given the tree $\mathcal S^5_{u}$ drawn in solid lines on the left
  of Figure~\ref{fig:distance-est}, we can read off the times
  $S_2,\dots,S_5$ which the tree spends with exactly $2,\dots,5$
  lines. These random variables are independent and each $S_k$ is
  exponentially distributed with mean $1/\binom k 2$. The black bullet
  indicates a recombination event with a mark between $u$ and $v$, and
  $\mathcal S^5_{v}$ can be read off by following the dashed line.

  On the right hand side of Figure~\ref{fig:distance-est} the trees
  $\mathcal S^5_{u}$ (black) and $\mathcal S^5_{v}$ (gray) are
  embedded into a common tree. The embedding of $\mathcal S^5_{v}$
  into the common tree is given by $x_i\mapsto x_i, i=1,2,3$ and
  $x_i\mapsto x_i', i=4,5$. Using \eqref{eq:grpr-grtv} and
  \eqref{eq:gr-tv-finite} below it is easily seen that if the trees
  $\mathcal S^5_{u}$ and $\mathcal S^5_{v}$ are equipped with uniform
  probability measure, we have
  \begin{align*}
    d_{\mathrm{GP}}(\mathcal S^5_{u}, \mathcal S^5_{v}) \leq
    d_{\mathrm{GTV}}(\mathcal S^5_{u}, \mathcal S^5_{v}) \leq 2/5.
  \end{align*}

  From the same embedding we see also that the Gromov-Hausdorff
  distance of $\mathcal S^5_{u}$ and $\mathcal S^5_{v}$, conditioned
  on $\mathcal S^5_{u}$ is bounded by the tree distance of
  $\{x_4',x_5'\}$ and $x_1$ which is twice the time of back
  coalescence of the dashed recombination line. Thus, in this
  particular case we have
  $d_{\mathrm{GH}}(\mathcal S^5_{u}, \mathcal S^5_{v})\le 2(S_2+ \dots
  + S_5)$.
  It could of course happen that the dashed line coalesces back after
  the time of the MRCA of $\mathcal S^5_{u}$. In this case the upper
  bound would be $2(S_1+ \dots + S_5)$, where $S_1$ is another
  independent exponentially distributed random variable with mean $1$.
\end{example}

\noindent
First, we explain a way to compute the Gromov total variation distance
explicitly.

\begin{remark}[How to compute $d_{\mathrm{GTV}}$\label{rem:TC_properties1}]
  If the spaces $X_1, X_2$ are finite with $\# X_1 = \#X_2=N$ and
  $\mu_i$ is the uniform distribution on $X_i, i=1,2$, then one can
  compute the Gromov total variation distance explicitly. There are
  isometric embeddings $\varphi_i$ of $X_i$ into a common finite
  metric space $(Z,d)$, $i=1,2$, so that $Z$ can be decomposed in
  three disjoint sets $Z = Z_1 \uplus Z_2 \uplus
  Z_{\textnormal{joint}}$ with the property $\varphi_i(X_i) = Z_i \cup
  Z_{\textnormal{joint}}$, $i=1,2$. The optimal embeddings are the
  ones for which $\#Z_{\textnormal{joint}}$ is maximal. Using such
  optimal embeddings we have
  \begin{multline}
    \label{eq:gr-tv-finite}
    d_{\mathrm{GTV}}(\smallx_1,\smallx_2)
    = \frac12 \mu_1(\varphi_1^{-1}(Z_1))+\frac12\mu_2(\varphi_2^{-1}(Z_2)) \\
    + \frac12 \sum_{z \in Z_{\textnormal{joint}}}
    \abs{\mu_1(\varphi_1^{-1}(\{z\}))-\mu_2(\varphi_2^{-1}(\{z\}))} =
    \frac{\# Z_1}{N} = \frac{\# Z_2}{N}.
  \end{multline}
\end{remark}

\noindent
As it turns out, to obtain upper bounds on Gromov total variation and
therefore also the Gromov-Prohorov distances it is helpful to
introduce yet another distance which is only well-defined on trees
from the processes $(\mathcal T_u^{\mathcal A_0})_{u\in\mathbbm G}$.

\begin{definition}[Auxiliary distance\label{def:aux}]
  Let $(\mathcal T_u^{N})_{u\in\mathbbm G}$ and $\mathcal A$ be as in
  Definition~\ref{def:ARG}. Recall the MRCA within $\mathcal T_u^{N}$,
  denoted by $\bullet_u$, from Definition~\ref{def:Tu}. For
  $u,v\in\mathbbm G$, we decompose $[N]$ into two subsets:
  \begin{enumerate}
  \item We denote by $[N]_{\bullet_u, v}$ the set of all $x\in [N]$
    such that the path within $\mathcal T_u^{N}$ from $x$ to
    $\bullet_u$ is \emph{not} hit by a splitting event marked with
    $U\in[u,v]$.
  \item Then $[N]_{\bullet_u, v}^c$ is the set of all $x\in [N]$ with
    such a splitting event on the path from $x$ to $\bullet_u$ in
    $\mathcal T_u^{N}$.
  \end{enumerate}
  We define
  \begin{align}
    d^{\bullet_u, v}_{\mathrm{aux}} (\mathcal T^{N}_{u}, \mathcal
    T^{N}_{v}) \coloneqq \frac{\# [N]_{\bullet_u, v}^c}{N}.
  \end{align}
\end{definition}

\begin{proposition}[Properties of the auxuliary distance]
  Let \label{P:aux} $d_{\mathrm{GTV}}$ be as in
  Definition~\ref{def:Grom}, $d_{\mathrm{aux}}$, as in
  Definition~\ref{def:aux}, $(\mathcal T_u^{N})_{u \in \mathbbm G}$ as
  in Definition~\ref{def:Tu} and
  $\mathpzc T_u^{N} = \overline{\mathcal T_u^{N}}$ as in
  Definition~\ref{def:Tumm}.
  \begin{enumerate}
  \item The distance $d_{\mathrm{aux}}$ is an upper bound for
    $d_{\mathrm{GTV}}$ in the sense that, for all $u,v\in\mathbbm G$
    \begin{align*}
      d_{\mathrm{GTV}} (\mathpzc T^{N}_{u}, \mathpzc T^{N}_{v}) \leq
      d^{\bullet_u, v}_{\mathrm{aux}} (\mathcal T^{N}_{u}, \mathcal
      T^{N}_{v})
    \end{align*}
  \item For $u,v,w \in\mathbbm G$ with $u<v<w$, the distances
    $d^{\bullet_v, u}_{\mathrm{aux}} (\mathcal T^{N}_{u}, \mathcal
    T^{N}_{v})$ and $ d^{\bullet_v, w}_{\mathrm{aux}} (\mathcal
    T^{N}_{v}, \mathcal T^{N}_{w})$ are independent given $\mathcal
    T^{N}_{v}$.
  \end{enumerate}
\end{proposition}

\begin{proof}
  1. Since the Gromov-Total-Variation distance is defined as the
  infimum over all embeddings into a common mm-space, it suffices to
  bound the Total-Variation distance on a concrete
  embedding. Therefore, we use Remark~\ref{rem:TC_properties1} and
  write $Z = Z_1 \uplus Z_2 \uplus Z_{\text{joint}}$ with
  $Z_{\text{joint}} = [N]_{\bullet_u, v}$ and $Z_1, Z_2$ are two
  copies of $[N]_{\bullet_u, v}^c$. Distances on $Z_{\text{joint}}$
  within $\mathcal T^{N}_{u}$ and $\mathcal T^{N}_{v}$ are identical
  by construction and we see from \eqref{eq:gr-tv-finite} that
  \begin{align*}
    d_{\mathrm{GTV}}(\mathpzc T^{N}_{u}, \mathpzc T^{N}_{v}) \leq
    d_{\mathrm{TV}}(\mathcal T^{N}_{u}, \mathcal T^{N}_{v}) \leq
    \frac{\#Z_1}{N} = \frac{\# [N]_{\bullet_u, v}^c}{N} =
    d^{\bullet_u}_{\mathrm{aux}} (\mathcal T^{N}_{u}, \mathcal
    T^{N}_{v}).
  \end{align*}
  2. From Definition~\ref{def:ARG'} and Remark~\ref{rem:PropS}, we see
  that the triple $(\mathcal T_u^{N}, \mathcal T_v^{N}, \mathcal
  T_w^{N})$ can be constructed starting with $\mathcal T_v^{N}$ (with
  $a=v, b=w$). The resulting graph $\mathcal G$ is then used to again
  use the same procedure with initial state $(\mathcal G, \mathcal
  T_v^{N})$ and use the same procedure (with $a=v$ and $b=u$) this
  time moving to the left of $v$. In total, this results in marking
  $\mathcal T_v^{N}$ at rates $\rho(w-v)$ and $\rho(v-u)$
  independently. Leaves which are marked at rate $\rho(w-v)$
  ($\rho(v-u)$) until $\bullet_v$ is hit, are elements of
  $[N]_{\bullet_v,u}^c$ ($[N]_{\bullet_v, w}$). Importantly, since the
  marking on $[u,v]$ and $[v,w]$ are independent (given $\mathcal
  T_v^{N}$) -- see Remark~\ref{rem:PropS} -- the claim follows.
\end{proof}

\begin{remark}[Bound of $d_{\mathrm{GTV}}$ by $d_{\mathrm{aux}}$ not
  sharp]
  It can be easily seen that the strict inequality
  $d_{\mathrm{GTV}} (\mathpzc T^{N}_{u}, \mathpzc T^{N}_{v}) <
  d^{\bullet_u, v}_{\mathrm{aux}} (\mathcal T^{N}_{u}, \mathcal
  T^{N}_{v})$
  is possible. This happens for instance if the dashed line in
  Figure~\ref{fig:distance-est} coalesces immediately with the same
  line.
\end{remark}

\begin{proposition}[Bounds on $d_{\mathrm{aux}}$\label{P:boundsaux}]
  Let $v,w\in\mathbbm G$ with $v<w$ and
  $(\mathcal T_u^{N})_{u\in\mathbbm G}$ be as in
  Definition~\ref{def:ARG'}. For $\mathcal T_v^{N}$, let
  $S_2, S_3,\dots$ be the duration for which $2,3,\dots$ lines are
  present in the tree.
    We have
    \begin{align}
      \label{eq:conddGTV-bound}
      \E\bigl[ d_{\mathrm{aux}}^{\bullet_v, w} (\mathcal T^N_{v},\mathcal
      T^N_{w}) | \mathcal T^N_{v}\bigr] \le \rho |w-v| \sum_{k=2}^N
      S_k.
    \end{align}
\end{proposition}

\begin{proof}
  For $x\in\mathcal A_0$, let $\mathcal L_x \subseteq \mathcal T_v^N$
  be the path from $x$ to $\bullet_v$. Its length is given by the tree
  height $L \coloneqq S_2 + \cdots + S_N$. 

  Given $\mathcal T_v^N$, $x \in [N]_{\bullet_v, w}$ with
  probability $1-e^{-\rho|w-v| L}$ by Remark~\ref{rem:PropS}.2. Hence,
  by exchangeability and the definition of $d_{\mathrm{aux}}$,
  \begin{align*}
    \E\bigl[ d_{\mathrm{aux}}^{\bullet_v, w} (\mathcal T^N_{v},\mathcal
    T^N_{w}) | \mathcal T^N_{v}\bigr] & = \frac 1N \sum_{x\in [N]}
    \mathbb P(x\in [N]_{\bullet_v, w}^c) = 1-e^{-\rho|w-v| L} \leq
    \rho|w-v| L.
  \end{align*}
\end{proof}

\subsection{Projective properties of the ARG}
It is well-known that Kingman's coalescent is projective in the sense
that the tree spanned by a sample of $n$ leaves from an $N$-coalescent
has the same distribution as an $n$-coalescent. The same holds for the
ARG as we will show next.

\begin{lemma}[Projectivity in $N$ for the $N$-ARG]
  Let $\mathbbm G = [a,b]$, $\rho>0$ and $\mathcal A$ be an $N$-ARG
  (with $\mathcal A_0 = [N] = \{1,\dots,N\}$). Let $\mathcal B \subseteq
  [N]$ with $\#\mathcal B = n$ and let $\pi_{\mathcal B} \mathcal A^N$
  be the random graph which arises from the particle system starting
  with particles $\mathcal B$ and following them along $\mathcal
  A$. (Upon coalescence events within $\mathcal A^N$, followed
  particles merge as well. If a coalescence event within $\mathcal A$
  only involves a single followed particle, continue to follow the
  coalesced particle. Splitting events hitting a followed particle are
  followed as well.) Then, $\pi_{\mathcal B}\mathcal A^N$ is an
  $n$-ARG.
\end{lemma}

\begin{proof}
  It suffices to consider the particle-counting process of
  $\pi_{\mathcal B}\mathcal A^N$ since the fine-structure of $\mathcal
  A^N$ is exchangeable. Clearly, pairs of particles coalesce at rate~1
  and every particle splits at the same rate $\rho(b-a)$. Hence, it
  coincides with the particle-counting process of $\mathcal A^n$ and
  we are done.
\end{proof}

\noindent
The projectivity of the $N$-ARG along the genome is stated next:

\begin{lemma}[Projectivity in $\mathbbm G$ of the
  $N$-ARG\label{l:proj2}]
  Let $\mathbbm G = [a,b]$, $\rho>0$, $\mathcal A^N(\mathbbm G)$ be an
  $N$-ARG and $\mathbbm H = [c,d]\subseteq \mathbbm G$. Let
  $\pi_{\mathbbm H} \mathcal A^N$ be the random graph which arises
  from the particle system starting with particles $[N]$ and following
  them along $\mathcal A$. (Upon coalescence events within $\mathcal
  A^N$, followed particles merge as well. If a coalescence event
  within $\mathcal A$ only involves a single followed particle,
  continue to follow the coalesced particle. Splitting events hitting
  a followed particle are followed as well if the mark falls in
  $\mathbbm H$.) Then, $\pi_{\mathbbm H}\mathcal A^N$ equals $\mathcal
  A^N(\mathbbm H)$ in distribution.
\end{lemma}

\begin{proof}
  It suffices to see that (i) the recombination events at loci in
  $[a,b]\setminus [c,d]$ split off ancestral material not in $[c,d]$
  and therefore don't change the genealogical trees in $\pi_{\mathbbm
    H}\mathcal A^N$ and (ii) coalescences with such lines don't appear
  in genealogical trees in $\mathbbm H$. Leaving out these
  recombination events hence leads to $\mathcal A^N(\mathbbm H)$,
  so the claimed equality follows.
\end{proof}

\section{Proof of Theorem~\ref{T1}}
The proof of Theorem~\ref{T1} requires three steps. First, in order to
obtain existence of limiting processes along subsequences of
$(\mathpzc T^N)_{N=1,2,\dots}$, we have to prove  (see Section~\ref{S:51})
\begin{align}\label{eq:toshow1}
  \text{The family $(\mathpzc T^N)_{N=1,2,\dots}$ is tight in $\mathcal
    D_{\mathbbm M}(\mathbbm G)$.}
\end{align}
Second, we show in Section~\ref{S:52}
\begin{align}\label{eq:toshow2}
  \text{\parbox{10cm}{For any limiting process $\mathpzc T$ along a
      subsequence of $(\mathpzc T^N)_{N=1,2,\dots}$, \eqref{eq:T11}
      holds.}}
\end{align}
This equation determines uniquely the finite-dimensional distributions
of $\mathpzc T$ since polynomials are separating; in particular, since
the right hand side of \eqref{eq:T11} does not depend on the
subsequence, uniqueness of the limiting process follows.  Third,
bounds on the variation process of $\mathpzc T$ are given in
Section~\ref{S:53} such that
\begin{align}\label{eq:toshow3}
  \text{\parbox{10cm}{The paths of $\mathpzc T$ are a.s.\ of finite
  variation with respect to Gromov total variation,
  Gromov-Prohorov and Gromov-Hausdorff metrics.}}
\end{align}

\subsection{Tightness in $\mathcal D_{\mathbbm M}(\mathbbm G)$}
\label{S:51}
In order to prove tightness in the sense~\eqref{eq:toshow1}, we rely
on Theorem~13.6 in \cite{Billingsley1999} (see also Theorem 3.8.8 in
\cite{EthierKurtz1986}), i.e.\ we will show that there is $C>0$ such
that
\begin{align}
  \label{eq:tight}
  \limsup_{N\to\infty} \mathbb E[d_{\mathrm{GTV}}(\mathpzc T_{-h}^N,
  \mathpzc T_{0}^N) \cdot d_{\mathrm{GTV}}(\mathpzc T_{0}^N, \mathpzc
  T_{h}^N)] \leq Ch^2.
\end{align}
Since $d_{\mathrm{GP}}\leq d_{\mathrm{GTV}}$, this implies tightness
with respect to the Gromov-weak topology.

We assume without loss of generality that the interval $[-h,h]$ is
contained in $\mathbbm G$. Note also that
$(\mathpzc T^N_u)_{u\in \mathbbm G}$ is stationary. We combine
Proposition~\ref{P:aux} and Proposition~\ref{P:boundsaux} and write,
for $S_k\sim \mathrm{Exp}\binom k 2$,
\begin{align*}
  \mathbb E[d_{\mathrm{GTV}}(\mathpzc T_{-h}^N, \mathpzc T_{0}^N) &
  \cdot d_{\mathrm{GTV}}(\mathpzc T_{0}^N, \mathpzc T_{h}^N)] \leq
  \mathbb E[d_{\mathrm{aux}}^{\bullet_0, -h}(\mathcal T_{-h}^N,
  \mathcal T_{0}^N) \cdot d_{\mathrm{aux}}^{\bullet_0, h}(\mathcal
  T_{0}^N, \mathcal T_{h}^N)] \\ & = \mathbb E\big[\mathbb
  E[d_{\mathrm{aux}}^{\bullet_0, -h}(\mathcal T_{-h}^N, \mathcal
  T_{0}^N)|\mathcal T_{0}^N] \cdot \mathbb
  E[d_{\mathrm{aux}}^{\bullet_0, h}(\mathcal T_{0}^N, \mathcal
  T_{h}^N)|\mathcal T_{0}^N]\big] \\ & \leq \rho h^2 \cdot \mathbb
  E\Big[\Big(\sum_{k=2}^N S_k\Big)^2\Big] \leq 11\rho h^2.
\end{align*}
The last estimate follows from the following elementary computation:
\begin{align*}
  \E\Bigl[\Bigl(\sum_{k=2}^N S_k)\Bigr)^2\Bigr] & = \sum_{k=2}^N
  \E[S_k^2]
  + 2 \sum_{2 \le k < \ell \le N} \E[S_k]\E[S_\ell] \\
  & \le \sum_{k=2}^\infty \E[S_k^2] + 2 \Bigl(\sum_{k=2}^\infty
  \E[S_k]\Bigr)^2 \\
  & = \sum_{k=2}^{\infty} \frac{8}{k^2(k-1)^2} + 2
  \Bigl(\sum_{\ell=2}^\infty \frac2{\ell(\ell-1)}\Bigr)^2 =
  8\Bigl(\frac{\pi^2}3 -3 \Bigr) + 8 < 11.
\end{align*}
Therefore, we have proved \eqref{eq:tight}.

\subsection{Finite-dimensional distributions}
\label{S:52}

Let $\Phi_i = \Phi^{n_i,\phi_i}, i=1,\dots,j$ be (bounded) polynomials
and $(\mathpzc T_u^N)_{u\in\mathbbm G}$ be as in Definition
\ref{def:Tumm} with $\mathpzc T_u^N = \overline{\mathcal T_u^N}$,
where $\mathcal T_u^N = ([N], r_u^N, \mu)$ is as in
Definition~\ref{def:Tu}. Then, $\mu$ is the uniform distribution on
$[N]$ and $r_u^N$ gives distances between elements of $[N]$ in
$\mathcal T_u^N$. Let 
\begin{align*}
  A_{n,N} = \bigcup_{1\leq i<j\leq n} \{\underline x\in [N]^n: x_i = x_j\}
  \subseteq [N]^n
\end{align*}
be the event that some entry in $\underline x \in [N]^n$ appears
twice. Then, we find that $\mu^{\otimes n}(A_{n,N}) = \mathcal O(1/N)$
(for fixed $n$) as $N\to\infty$. Therefore, by construction, setting
$\underline x_{kl} = (x_k,\dots,x_l)$ and $r(\underline x_{kl},
\underline x_{kl}) = (r(x_i, x_j))_{k\leq i,j\leq l}$, $\overline
n_0=0$, $\overline n_i = n_1 + \cdots + n_i$,
\begin{align*}
  \lim_{N\to\infty} \mathbb E[\Phi_1( & \mathpzc T^N_{u_1})\cdots
  \Phi_j(\mathpzc T^N_{u_j})] = \lim_{N\to\infty} \mathbb E\Big[
  \prod_{i=1}^j \int \mu^{n_i}(d\underline x_{1n_i})
  \phi_i(r_{u_i}^N(\underline x_{1n_i}, \underline x_{1n_i}))\Big] \\
  & = \lim_{N\to\infty} \mathbb E\Big[ \int \mu^{n}(d\underline
  x_{1n}) \prod_{i=1}^j \phi_i(r_{u_i}^N(\underline
  x_{\overline{n}_{i-1},{\overline{n}_{i}}}, \underline
  x_{\overline{n}_{i-1},{\overline{n}_{i}}}))\Big] \\ & =
  \lim_{N\to\infty} \mathbb E\Big[ \int \mu^{n}(d\underline x_{1n})
  \indset{A_{n,N}^c}\prod_{i=1}^j \phi_i(r_{u_i}^N(\underline
  x_{\overline{n}_{i-1},{\overline{n}_{i}}}, \underline
  x_{\overline{n}_{i-1},{\overline{n}_{i}}}))/\mu^n(A_{n,N}^c)\Big] \\
  & =\mathbb E[\phi_1(\underline{\underline R}_1)
  \cdots\phi_j(\underline{\underline R}_j)].
\end{align*}

\subsection{Finite variation}
\label{S:53}

In this subsection we prove the last part of Theorem~\ref{T1}, namely
that the paths of the process
$\mathpzc T = (\mathpzc T_u)_{u\in\mathbbm G}$ are of finite variation
with respect to Gromov total variation, Gromov-Prohorov metric and
Gromov-Hausdorff (semi-)metric on $\mathbbm M$. Recall the definition
of variation with respect to a metric in \eqref{eq:TV-f}. First we
show that for
$d \in \{d_{\mathrm{GTV}},d_{\mathrm{GP}},d_{\mathrm{GH}}\}$ that for
$u,v \in \mathbbm G$, $u < v$ there is a positive finite constant
$C=C(\rho)$ such that
\begin{align}
  \label{eq:fv1}
       \E[ d(\mathpzc T_{u}^N,\mathpzc T_{v}^N)] \le C (v-u).
\end{align}
Once this is proven for a metric $d$, for any interval
$[a,b] \subset \mathbbm G$ and a partition
$a=u_0 < u_1 < \dots < u_k =b$ of that interval we have by the first
part of Theorem~\ref{T1}
\begin{align}
  \label{eq:GP-fv-bound}
  \E[ d(\mathpzc T_{u_i},\mathpzc T_{u_{i-1}})] \le
  \limsup_{N \to \infty}   \E[ d(\mathpzc T_{u_i}^N,\mathpzc
  T_{u_{i-1}}^N)] \le C (u_i-u_{i-1}).
\end{align}
Then it follows easily
\begin{align*}
  \E\Bigl[\sum_{i=1}^k d (\mathpzc T_{u_i},\mathpzc T_{u_{i-1}})
  \Bigr] \le C(b-a).
\end{align*}
Since the right hand side does not depend on particular partition of
$[a,b]$, this shows that the variation of $\mathpzc T$ with respect to
$d$ has finite expectation on finite intervals. Thus, the paths of
$\mathpzc T$ are almost surely of finite variation with respect $d$.

For Gromov total variation and Gromov-Prohorov metrics \eqref{eq:fv1}
follows from $d_{\mathrm{GP}} \le d_{\mathrm{GTV}} \le d_{\mathrm{aux}}$
and \eqref{eq:conddGTV-bound}. For the Gromov-Hausdorff metric
\eqref{eq:fv1} is shown in the following lemma.
\begin{lemma}
   There is a positive finite constant $C$ independent of $N$ so that
   for any $u,v \in \mathbbm G$, $u<v$ we have
   \begin{align}
     \label{eq:GH-fv-bound}
     \E[ d_{\mathrm{GH}}(\mathcal T_{u}^N,\mathcal T_{v}^N)]
     \le C \rho(v-u).
   \end{align}
\end{lemma}

\begin{proof}
  Given the tree $\mathcal T_{u}^N$ as before we denote by
  $S_2,\dots,S_N$ the time for which exactly $2,\dots,N$ lines are
  present in the tree (cf.\ Figure~\ref{fig:distance-est}). The random
  variables $S_2,\dots,S_N$ are independent and $S_k$ is exponentially
  distributed with mean $1/\binom{k}2$.

  Along the branches of $\mathcal T_{u}^N$ recombination events occur
  at rate $\rho(v-u)$. When a recombination event occurs at level $k$,
  that is during the period of time with exactly $k$ lines in the tree
  $\mathcal T_{u}^N$, then the resulting extra line coalesces back
  into the tree $\mathcal T_{u}^N$ at some time after the
  recombination, that is at level $\ell$ for some
  $1 \leq \ell \leq k$. We also need to consider the level $\ell=1$
  because it might be the case that the extra line coalesces back into
  the tree $\mathcal T_{u}^N$ after all lines of $\mathcal T_{u}^N$
  have coalesced with each other.

  Let $S_1$ be exponentially distributed with mean $1$. Furthermore,
  for $k \ge 2$ and $1 \leq \ell \leq k$ let $A_{k,\ell}$ be the event
  that along the $k$ branches of $\mathcal T_{u}^N$ during time $S_k$,
  at least one recombination event occurs that separates the trees
  $\mathcal T_{u}^N$ and $\mathcal T_{v}^N$, and the resulting extra
  line coalesces back into the tree $\mathcal T_{u}^N$ during time
  $S_\ell$. Figure~\ref{fig:distance-est} shows an example of the
  event $A_{3,2}$.

  Then, ignoring the probability of no coalescence during $S_k$ (hence
  bounding this probability from above by~1),
  \begin{align*}
    \P \bigl[ A_{k,\ell} | \mathcal T_{u}^N \bigr] & \leq (1 -
    e^{-\rho (v-u) kS_k}) \prod_{m=\ell+1}^{k-1} e^{-mS_m}
    (1-e^{-\ell S_\ell}) \\
    & \leq \rho (v-u) kS_k \ell S_\ell \prod_{m=\ell+1}^{k-1}
    e^{-mS_m}.
  \end{align*}
  Note that
  \begin{align*}
    \E \Big[ \prod_{m=\ell+1}^{k-1} e^{-mS_m}\Big]
    & = \prod_{m=\ell+1}^{k-1} \frac{1}{1 + 2/(m-1)} = \prod_{m=\ell}^{k-2} \frac{1}{1 + 2/m} \\
    & = \exp\Big( - \sum_{m=\ell}^{k-2} \log(1+2/m)\Big) \leq
    \exp\Big( - \sum_{m=\ell}^{k-2} 1/m\Big) \\
    & \leq \exp\Big(-\int_\ell^{k-2} \tfrac 1x dx\Big) =
    \frac{\ell}{k-2}.
  \end{align*}
  Furthermore, given $\mathcal T_{u}^N$, on the event $A_{k,\ell}$ we
  have
  \begin{align*}
    d_{\mathrm{GH}}(\mathcal T_{u}^N,\mathcal T_{v}^N) \le 2
    \sum_{j=\ell}^N S_j.
  \end{align*}
  It follows that for some $C_0, C_1, C_2, C_3>0$, which don't depend
  on $N$, $\rho$, $u$ and $v$,
  \begin{align*}
    \limsup_{N\to\infty}\E[ d_{\mathrm{GH}}(\mathcal T_{u}^N,\mathcal T_{v}^N)]
    & \le \limsup_{N\to\infty} \sum_{\ell\leq k\leq N} \E\Bigl[
      2\sum_{j=\ell}^N S_j; A_{k,\ell} \Bigr] \\
    & \leq \limsup_{N\to\infty} C_0 \sum_{\ell\leq k\leq N}
      \sum_{j=\ell}^N \E[ S_j ]\P \bigl[ A_{k,\ell} \bigr] \\
    & \le \limsup_{N\to\infty}C_1\rho (v-u) \sum_{\ell\leq k\leq N}
      \sum_{j=\ell}^N k\ell \frac{\ell/k}{\binom{j}{2}
      \binom{k}{2}\binom{\ell}{2}} \\
    & \leq \limsup_{N\to\infty}C_2\rho(v-u) \sum_{\ell\leq k\leq N}
      \frac{1}{\ell k^2} = C_3 \rho (v-u),
  \end{align*}
  which shows the assertion.
\end{proof}

\section{Proof of Theorem~\ref{T2}}
Theorem~\ref{T2} claims that correlations between trees $\mathpzc T_0$
and $\mathpzc T_v$ decrease with $\mathcal O(1/v^2)$. Such
correlations come with coalescence times present in $\mathpzc T_0$
and $\mathpzc T_v$. Before we come to the proof of Theorem~\ref{T2},
we study such joint coalescences.

\subsection{Covariances of coalescence times}
\begin{lemma}[Covariance of distances at $u=0$ and $u=v$\label{l:T0v}]
  Let $\Phi = \Phi^{2,\phi}$. Then,
  \begin{align}\label{eq:T0v}
    \mathbb P(\Phi^{2,\phi}(\mathpzc T_0) = \Phi^{2,\phi}(\mathpzc
    T_v) \text{ for all $\phi$}) = \frac{2}{9 + 13 \rho v +
      2\rho^2v^2}.
  \end{align}
\end{lemma}

\begin{proof}
  From Theorem~\ref{T1}, we see that the left hand side of
  \eqref{eq:T0v} is given as follows: Let $\mathcal A^4$ be an ARG
  starting with four lines, $R_{12,0}$ be the distance of the pair
  $1,2$ at $u=0$ and $R_{34, v}$ the distance of $3,4$ at $u=v$. Then,
  $\Phi^{2,\phi}(\mathpzc T_0) = \Phi^{2,\phi}(\mathpzc T_v)$ for all
  $\phi$ has the same probability as $R_{12,0}=R_{34, v}$ (recall that
  $\phi$ is a function in the definition of polynomials is a function
  of pairwise distances; see \eqref{eq:polynom}).

  Hence, the LHS of \eqref{eq:T0v} equals $x$, where
  \begin{align*}
    x & = \mathbb P(R_{12,0} = R_{34,v}) \\
    y & = \mathbb P(R_{12,0} = R_{23,v}) \\
    z & = \mathbb P(R_{12,0} = R_{12,v}) \\
  \end{align*}
  The first event in the ARG starting in four lines, can be:
  \begin{itemize}
  \item[(i)] coalescence of one of the pairs (1,3), (2,3), (1,4),
    (2,4)
  \item[(ii)] coalescence of one of the pairs (1,2), (3,4)
  \item[(iii)] Some recombination event.
  \end{itemize}
  In case (iii), the probability $x$ is not changed after the
  recombination event, in case (ii), there is no way that the event
  $R_{12,0}=R_{34, v}$ (hence has probability~0). In case (i),
  however, the probability is the same as in an ARG with three lines,
  that lines 1 and 2 at locus 0 coalesce at the same time as lines 2
  and 3 at locus v. This probability is defined to be $y$. Similar
  arguments for a first-event-decomposition in the probabilities $y$
  and $z$ lead to
  \begin{equation}
    \label{eq:resta}
    \begin{aligned}
      x & = \frac{2}{3} y + \frac 13 \cdot 0,\\
      y & = \frac{\rho v}{\rho v + 3} x + \frac{1}{\rho v + 3} z + \frac{2}{\rho v + 3}\cdot 0,\\
      z & = \frac{2\rho v}{2\rho v + 1}y + \frac{1}{2\rho v +1}\cdot 1.
    \end{aligned}
  \end{equation}
  Solving this linear system gives the result.
\end{proof}

\noindent
Using the last lemma, we immediately obtain another useful result.

\begin{corollary}[Samples of size $n\geq 2$\label{cor:T0v}]
  Let $n\geq 2$, $\mathcal A^{2n}$ be a $2n$-ARG, $R_{ij,u}$ be the
  distance of the pair $i,j$ at position $u$ for $u\in\{0,v\}$. Then,
  \begin{align*}
    \mathbb P(R_{ij,0} = R_{k \ell,v} \text{ for some } 1\leq i<j \leq
    n;\;
    & n+1\leq k<\ell\leq 2n) \\
    & \leq \binom{n}{2}^2 \frac{2}{9 + 13 \rho v + 2\rho^2v^2}.
  \end{align*}
\end{corollary}

\subsection{An auxiliary random graph}
For the proof of Theorem~\ref{T2}, we recall the $2n$-ARG $\mathcal
A^{2n}$ for loci $v\in\{0,u\}$. We let $R_{ij,v}$ be the distance
between $i,j$ at locus $v$ for $v\in\{0,u\}$. We set $\mathcal T_0 =
\mathcal T_0^{\{1,\dots,n\}}$ and $\mathcal T_v = \mathcal
T_v^{\{n+1,\dots,2n\}}$ (recall the notation from
Definition~\ref{def:Tu}). The following events can happen:
\begin{enumerate}
\item ``Intra-tree'' coalescence events: If in $\mathcal A^{2n}$ two
  particles coalesce and both particles belong to $\mathcal T_v$, then
  the total number of particles \emph{and} the number of particles in
  $\mathcal T_v$ decreases, $v\in\{0,u\}$.
\item ``Inter-tree'' coalescence events: If two particles (in
  $\mathcal A^{2n}$) coalesce and one of the particles is present in
  $\mathcal T_u \setminus \mathcal T_0$, and the other is present in
  $\mathcal T_0 \setminus \mathcal T_u$, then the total number of
  particles decreases but the numbers of particles within $\mathcal
  T_0$ and $\mathcal T_u$ are preserved.
\item ``Splitting recombination'' events: If a particle, present in
  the overlap of the trees $\mathcal T_u \cap \mathcal T_0$,
  recombines with mark $U\in[0,u]$, then the particle splits in two
  new particles, one present in $\mathcal T_u\setminus\mathcal T_0$,
  the other one present in $\mathcal T_0\setminus\mathcal T_u$.
\end{enumerate}
We call a branch in $\mathcal A^{2n}$ a \emph{single line} if it
belongs to $(\mathcal T_0\setminus\mathcal T_u) \cup (\mathcal
T_u\setminus\mathcal T_0)$, whereas branches in $\mathcal
T_0\cap\mathcal T_u$ are called \emph{double lines}. Intra-coalescence
occur simultaneously within $\mathcal T_0$ and $\mathcal T_u$ if two
double lines coalesce. These are the events that make $\mathcal T_0$
and $\mathcal T_v$ dependent. We will call such an event \emph{joint
  coalescence}.

~

\noindent
We now define a random graph $\widehat {\mathcal A}^{2n}$ based on
$\mathcal A^{2n}$ such that we can couple $\mathcal T_0$ and $\mathcal
T_u$ with two independent trees $\widehat{\mathcal T}_0$ and
$\widehat{\mathcal T}_u$, both having the distribution of a Kingman's
$n$-coalescent; see Lemma~\ref{l:aux}.

\begin{definition}[An auxiliary random graph]
  Define a random graph $\widehat {\mathcal A}^{2n}$, from which two
  trees $\widehat{\mathcal T}_0$ and $\widehat{\mathcal T}_u$ can be
  read off as in Definition~\ref{def:Tu}, as follows: Starting with
  $2n$ single lines, where $1,\dots,n \in \widehat{\mathcal T}_0
  \setminus \widehat{\mathcal T}_u$ and
  $n+1,\dots,2n\in\widehat{\mathcal T}_u \setminus \widehat{\mathcal
    T}_0$, the dynamics of the lines in $\widehat{\mathcal A}^{2n}$
  are as follows (see Figure~\ref{fig:mod-arg}):
  \begin{enumerate}[(i)]
  \item Each pair of single lines coalesces at rate $1$. The result
    can be a single line (if both lines belong to $\widehat{\mathcal
      T}_0$ or both to $\widehat{\mathcal T}_u$) or a double line.
  \item Each pair of lines where one is a single line and the other a
    double line coalesces at rate $1$. The resulting line is a double
    line.
  \item Each double line splits at rate $\rho u$ into two single
    lines.
  \item Between each pair of double lines there is a
    \emph{coalescence/splitting event} at rate $2$. This event
    produces a double line and a single line. With probability $1/2$
    the resulting single line is in $\widehat{\mathcal T}_0$ or in
    $\widehat{\mathcal T}_u$, respectively.
  \end{enumerate}
\end{definition}

\begin{remark}[Properties of $\widehat{\mathcal T}_0$ and
  $\widehat{\mathcal T}_u$]
  Note that the above dynamics in $\widehat{\mathcal A}^{2n}$ are the
  same as in $\mathcal A^{2n}$ except for the coalescence/splitting
  event described in (iv).  The corresponding event in $\mathcal
  A^{2n}$ was called joint coalescence above. In particular, we remark
  that we can perfectly couple $\mathcal A^{2n}$ with
  $\widehat{\mathcal A}^{2n}$ until the first coalescence/splitting
  event occurs.
\end{remark}

\begin{lemma}[Properties of $\widehat{\mathcal A}^{2n}$\label{l:aux}]
  We note the following properties of $\widehat{\mathcal A}^{2n}$:
  \begin{enumerate}
  \item $\widehat{\mathcal T}_0$ and $\widehat{\mathcal T}_u$ are
    independent and distributed as $n$-coalescents.
  \item If we couple $\mathcal A^{2n}$ and $\widehat{\mathcal A}^{2n}$
    until the first event (iv) happens, and let them evolve
    independently otherwise, then
    \begin{align*}
      \{\text{no event (iv) happens}\}
      \subseteq \{\mathcal T_0 = \widehat{\mathcal T}_0\}\cap
      \{\mathcal T_u = \widehat{\mathcal T}_u\}.
    \end{align*}
  \item For $n \ge 2$, there is $C=C(n)>0$ such that
    \begin{align*}
       \mathbbm P(\text{no event (iv) happens}) \leq C/ (\rho^2 u^2).
    \end{align*}
  \end{enumerate}
\end{lemma}

\begin{proof}
1. Obviously in $\widehat{\mathcal A}^{2n}$ each pair of lines in
$\widehat{\mathcal T}_0$ coalesces at rate $1$ and also each pair of
lines in $\widehat{\mathcal T}_u$ coalesces at rate $1$, so that both
trees are Kingman's coalescents, and the trees are independent by
construction.

\noindent
2. Denoting by $A$ the event that at a coalescence/splitting event in
$\widehat{\mathcal A}^{2n}$ occurs, we have
$\mathcal A^{2n} = \widehat{\mathcal A}^{2n}$ on $A^c$ by
construction.

\noindent
3. By construction, $A$ occurs at rate $2$ for every pair of lines
within $\mathcal Z_{\mathrm{joint}}$. For non-negative integers $a$,
$b$ and $c$ we indicate by $\mathbbm P_{abc}$ computations of
probabilities within $\widehat{\mathcal A}^{2n}$ with start in
\begin{itemize}
\item $a$ single lines within $\widehat{\mathcal T}_0$,
\item $b$ double lines,
\item $c$ single lines within $\widehat{\mathcal T}_u$.
\end{itemize}
Then, for
\begin{align*}
  x = \mathbbm P_{202}(A), \qquad y = \mathbbm P_{111}(A), \qquad
  z = \mathbbm P_{020}(A),
\end{align*}
we obtain the same set of equations as in~\eqref{eq:resta} with the
last one replaced by
\begin{align*}
   z  = \frac{2\rho u}{2\rho u + 2}y + \frac{2}{2\rho u +2}\cdot 1.
\end{align*}
Solving this system gives
\begin{align*}
   x = \frac{2}{9 + 7\rho u + \rho^2 u^2},
\end{align*}
which shows the assertion for $n=2$. As in Corollary~\ref{cor:T0v}, we
obtain for all $n=2,3,\dots$
\begin{align*}
  \P(\text{some event (iv) happens}) \leq {\binom n 2}^2 \frac{2}{9 + 7\rho u + \rho^2 u^2}
\end{align*}
which concludes the proof.
\end{proof}

\begin{figure}
  \centering
  \begin{tikzpicture}[xscale=0.9,yscale=0.9]
    \draw[very thick,black] (0,0) -- (0,1) -- (1,1) -- (1,0)
                            (0.5,1) -- (0.5,2.8)--(0,2.8)--(0,4) --
                            (0.5,4) -- (0.5,5) --(1.5,5) -- (1.5,5.6)
                            -- (1,5.6) -- (1,7)
                            (2,0) -- (2,1.5) -- (2.45,1.5) --
                            (2.45,4.1) -- (3.5,4.1)--(3.5,5) -- (1.5,5);

    \draw[very thick,lightgray] (3,0) -- (3,1.5) -- (2.55,1.5) --
                                (2.55,4) -- (3.6,4) -- (3.6,6.5) -- (3.1,6.5)
                                (4,0)  -- (4,2.4) -- (2.69,2.4)
                                (2.31,2.4)--(0.6,2.4) --(0.6,2.8) --
                                (1,2.8) -- (1,4) -- (0.6,4) -- (0.6,4.9) --
                                (1.6,4.9) -- (1.6,5.6) -- (2.6,5.6) --
                                (2.6,6.5) -- (3.1,6.5)--(3.1,7)
                                (5,0) -- (5,4) -- (3.6,4);
    \draw[very thick,lightgray] (2.3,2.4) to[out=45,in=135] (2.7,2.4);

    \draw (0.5,0.7) node {(i)};
    \draw (2.5,1.2) node {(i)};
    \draw (0.5,3.1) node {(iii)};
    \draw (3.6,3.7) node {(ii)};

    \draw[dashed] (2.5,4.95) ellipse (42pt and 6pt);
    \draw (4.4,5) node {(iv)};

  \end{tikzpicture}
  \caption{Reading off trees at different loci starting with disjoint
    sets of leaves from modified ARG $\widehat A^{2n}$ with $n=3$.
    Some of the events are annotated according to the description. The
    dashed ellipsis encloses the event which is not possible in the
    original ARG, cf.\ Figure~\ref{fig:arg-neigb-trees}.}
  \label{fig:mod-arg}
\end{figure}

\subsection{Proof of Theorem~\ref{T2}}
Let $\Psi = \Psi^{n,\psi}$ and $\Phi = \Phi^{n,\phi}$. According to
Theorem~\ref{T1}, we need to consider a $2n$-ARG $\mathcal A^{2n}$ and
let $R_{ij,v}$ be the distance between $i,j$ at locus $v$ for
$v\in\{0,u\}$. Writing $\underline{\underline R}_0 \coloneqq
(R_{ij,0})_{1\leq i,j\leq n}$, $\underline{\underline R}_v \coloneqq
(R_{ij,v})_{n+1\leq i,j\leq 2n}$, Theorem~\ref{T1} gives
\begin{align}\label{eq:PhiPsiDep}
  \mathbb{COV}[\Psi(\mathpzc T_0),\Phi(\mathpzc T_u)] =
  \mathbb{COV}[\psi(\underline{\underline R}_0),
  \phi(\underline{\underline R}_u)].
\end{align}
Let $\widehat{\mathcal T}_0, \widehat{\mathcal T}_u$ be as in
Lemma~\ref{l:aux}, which are coupled with $\mathcal T_0,
\mathcal T_u$ before the first coalescence/splitting event
happens. Let $\underline{\underline{\widehat R}}_0$ and
$\underline{\underline {\widehat R}}_u$ be the (finite) distance
matrices that correspond to $\widehat{\mathcal T}_0$ and
$\widehat{\mathcal T}_u$. Slightly abusing the notation we write
\begin{align*}
  \psi_0 = \psi(\underline{\underline R}_0),\quad   \phi_u =
  \phi(\underline{\underline R}_u), \quad  \widehat\psi_0 =
  \psi(\underline{\underline{\widehat R}}_0)\quad \text{and} \quad
  \widehat\phi_u = \phi(\underline{\underline{\widehat R}}_u).
\end{align*}
Denoting by $A$ the event that a coalescence/splitting event in
$\widehat{\mathcal A}^{2n}$ occurs, we have using Lemma~\ref{l:aux}
\begin{align*}
  \E[\psi_0 \phi_u]
  & = \E[\psi_0 \phi_u \indset{A^c}] + \E[\psi_0 \phi_u \indset{A}] \\
  & = \E[\widehat \psi_0 \widehat  \phi_u \indset{A^c}] + \E[\psi_0
    \phi_u \indset{A}]\\
  & = \E[\widehat \psi_0 \widehat  \phi_u] - \E[\widehat \psi_0
    \widehat  \phi_u  \indset{A}] + \E[\psi_0 \phi_u \indset{A}]
  \\
  & =  \E[\widehat \psi_0] \E[\widehat  \phi_u] - \E[\widehat \psi_0
    \widehat  \phi_u  \indset{A}] + \E[\psi_0 \phi_u \indset{A}]
\\
  & =  \E[\psi_0] \E[\phi_u] - \E[\widehat \psi_0
    \widehat  \phi_u  \indset{A}] + \E[\psi_0 \phi_u \indset{A}].
\end{align*}
It follows now that for $ C = 2 \mathbbm P(A)$, we have
\begin{align*}
 \Abs{ \E[\psi_0 \phi_u] -   \E[\psi_0]\E[\phi_u]} \le
 C \norm{\psi}_\infty \norm{\phi}_\infty,
\end{align*}
which, in view of Lemma~\ref{l:aux}.3.\ shows the assertion of
Theorem~\ref{T2} .

\paragraph{Acknowledgments}
This research was supported by the DFG through the priority program
1590, and in particular through grant Pf-672/6-1 to PP.

\bibliographystyle{chicago}

\end{document}